\documentclass{article}

\PassOptionsToPackage{numbers, compress}{natbib}

\usepackage[main,final]{neurips_2026}

\usepackage[utf8]{inputenc} 
\usepackage[T1]{fontenc}    
\usepackage{hyperref}       
\usepackage{url}            
\usepackage{booktabs}       
\usepackage{amsfonts}       
\usepackage{nicefrac}       
\usepackage{microtype}      
\usepackage{xcolor}         

\usepackage{amsmath,amssymb,bm,mathtools}
\usepackage{amsthm}
\usepackage{multirow}
\usepackage{graphicx}
\usepackage{placeins}

\graphicspath{{fig/}}
\providecommand{\E}{\mathbb{E}}

\newcommand{\norm}[1]{\left\lVert #1 \right\rVert_2}

\theoremstyle{plain}
\newtheorem{proposition}{Proposition}[section]
\theoremstyle{remark}

\title{Born-Preconditioned Learning\\for Iterative Neural PDE Solvers}

\author{%
  Juntao Wang\textsuperscript{1,3}\quad
  Jiwei Jia\textsuperscript{1,3}\thanks{Corresponding authors: \texttt{jiajiwei@jlu.edu.cn} (Jiwei Jia); \texttt{xinliang.liu@ouc.edu.cn} (Xinliang Liu).}\quad
  Xinliang Liu\textsuperscript{2,3}\footnotemark[1] \\
  \textsuperscript{1}School of Mathematics, Jilin University, Changchun 130012, China\\
  \textsuperscript{2}Ocean University of China, Qingdao 266100, China\\
  \textsuperscript{3}Shenzhen Loop Area Institute, Shenzhen 518038, China
}

\hypersetup{
  hidelinks,
  pdftitle={Born-Preconditioned Learning for Iterative Neural PDE Solvers},
  pdfauthor={Juntao Wang, Jiwei Jia, Xinliang Liu}
}

\begin{document}

\maketitle

\begin{abstract}
Simulating high-frequency waves in spatially varying media remains challenging for both classical iterative methods and neural PDE solvers.
We propose Neural Preconditioned Born Series (NPBS), a framework for learning iterative neural PDE solvers.
Our guiding idea is to use a simple reference preconditioner to help learn a more effective neural preconditioner.
Building on the Convergent Born Series (CBS), NPBS feeds a preconditioned residual to a neural network that predicts a solution update.
We train the network to reduce the preconditioned residual after each update.
The same reference preconditioner is used for the network input and the loss, giving a metric-matched training objective.
Experiments on heterogeneous Helmholtz problems show reduced iteration counts across neural architectures and larger gains at higher wavenumbers.
Within a Krylov solver, NPBS achieves the lowest wall-clock time in the evaluated solver comparison.
The framework also extends naturally to three-dimensional problems.
The same metric-matched formulation also improves convergence on convection--diffusion--reaction systems and Newton linear systems for nonlinear PDEs, indicating that NPBS provides a general framework for learning neural preconditioners.
Code is available at \url{https://github.com/wangjt23/NeuralPreconditionedBornSeries}.
\end{abstract}

\section{Introduction}
Efficient solution of wave equations in heterogeneous media is central to a broad range of engineering and scientific applications, from seismic inversion and medical ultrasound imaging to electromagnetic scattering and photonic design. Many of these problems reduce to the heterogeneous Helmholtz equation, whose high-wavenumber and high-contrast regimes remain challenging for both classical and learning-based solvers~\cite{gander2019class, zeng2025openbreastus}.
The main obstacles are well known: standard discretizations suffer from pollution errors~\cite{babuska1997pollution}, multilevel methods face coarse-grid dispersion mismatch~\cite{stolk2016dispersion}, and indefiniteness can undermine relaxation and coarse correction~\cite{ernst2011difficult}. Classical approaches address these issues through shifted-Laplacian preconditioning~\cite{erlangga2006novel,erlangga2004class}, wave-ray multigrid~\cite{brandt1997wave}, Convergent Born Series (CBS)~\cite{osnabrugge2016convergent}, and domain decomposition~\cite{chen2013source,graham2020domain}. Despite these advances, robust and efficient solution of heterogeneous Helmholtz problems remains difficult.

Learning-based PDE solvers have developed rapidly in recent years. Physics-informed neural networks (PINNs)~\cite{raissi2019physics} encode governing equations and boundary conditions directly into the training objective, but often require a separate optimization for each new instance, limiting cross-instance generalization and amortized efficiency~\cite{sun2020surrogate}. Neural operators address this limitation by learning mappings between function spaces, allowing one model to solve families of PDEs parameterized by coefficients, sources, or geometries. Representative works include DeepONet~\cite{lu2021learning}, FNO~\cite{li2020fourier,li2023fourier}, and MgNO~\cite{he2024mgno}, among many others~\cite{kovachki2023neural,cao2021choose,cao2025spectral,MELCHERS2026118893}.
While neural operators greatly improve scalability, end-to-end operator learning remains difficult for highly oscillatory solutions. Because of the spectral bias of neural networks~\cite{rahaman2019spectral,liu2024mitigating}, high-frequency components are typically much harder to learn accurately than low-frequency ones. For high-frequency Helmholtz equations, this makes one-shot wavefield prediction particularly challenging~\cite{he2025self}.

A promising alternative is to embed a learned module inside an iterative solver, using the network as a residual-to-correction map rather than as a direct solution predictor~\cite{kopanivcakova2025leveraging,rudikov2024neural,zhang2024blending}. Recent learned iterative solvers have shown that this strategy can reduce iteration counts or wall-clock time relative to classical baselines~\cite{xie2025mgcfnn}, making this direction especially promising for challenging Helmholtz problems.
Our work follows this direction, but asks a more specific question: what residual should a correction network receive, and what residual norm should its training loss reduce?

We answer this question by proposing \emph{Neural Preconditioned Born Series} (NPBS), a learned iterative preconditioning framework inspired by the CBS.
For Helmholtz problems, the Born residual is exactly the residual of a shifted-Laplacian left-preconditioned system.
This equivalence identifies the Born-preconditioned residual coordinates as a natural state space for learned iterative correction.
NPBS builds on this view by replacing the diagonal CBS correction with a neural residual-to-correction map acting on Born-preconditioned residuals.
To train this map, we minimize the relative norm of the Born-preconditioned residual after one correction.
This residual forms the network input at the next NPBS step.
Let $G_\eta$ denote the Green operator of the damped reference operator.
We introduce the residual metric induced by $G_\eta$:
\begin{equation}
\|\bm r\|_{R_\eta} := \|G_\eta \bm r\|_2,
\qquad
R_\eta := G_\eta^\ast G_\eta,
\end{equation}
and derive a metric-matched Born-series objective $\mathcal{L}_{\mathrm{bs}}^{R_\eta}$ that aligns training and inference within the same preconditioned geometry.
This metric reweights residual modes through the reference inverse, whose response can be highly nonuniform near resonance~\cite{ernst2011difficult,erlangga2004class}.
The same construction applies to convection--diffusion--reaction (CDR) systems and Newton linear systems arising from nonlinear PDEs~\cite{hao2024newton,lee2025neural}.

Our contributions are summarized as follows:
\begin{enumerate}
\item \textbf{A learned Born-preconditioned iterative framework.}
We introduce NPBS, which lifts the diagonal CBS preconditioner to an architecture-agnostic learned residual-to-correction map operating on Born-preconditioned residuals.

\item \textbf{A metric-matched training principle induced by the reference Green operator.}
By identifying Born residuals with shifted-Laplacian left-preconditioned residuals, we derive the induced metric $R_\eta = G_\eta^\ast G_\eta$ and the objective $\mathcal{L}_{\mathrm{bs}}^{R_\eta}$, making residual-metric matching the principle that aligns training with preconditioned inference.

\item \textbf{Robustness to ill-conditioning and algorithmic generalization.}
The two-dimensional FNO study shows iteration reductions of $1.2\times$--$1.9\times$ over direct learning as wavenumber increases, at $0.7\%$ additional training cost. Further tests cover four additional neural backbones, three-dimensional Helmholtz problems, CDR systems, and Newton linear systems. FGMRES with NPBS attains the lowest solve time among the evaluated two-dimensional classical and learned solvers.
\end{enumerate}

The paper is organized as follows. Section~\ref{sec:related_work} positions NPBS relative to neural preconditioners and Born-inspired neural solvers. Section~\ref{sec:born_identity} establishes the algebraic equivalence between Born-series residuals and shifted-Laplacian left preconditioning. Section~\ref{sec:npbs} introduces the NPBS iteration and derives the metric-matched training objective. Section~\ref{sec:numerical_results} reports experiments on heterogeneous Helmholtz benchmarks across difficulty levels and backbone architectures, with extensions to CDR and nonlinear PDEs.

\section{Related Work}
\label{sec:related_work}

\paragraph{Neural preconditioners and learned iterative solvers.}
Recent learning-based solvers increasingly use neural networks as components of an iterative numerical method rather than as standalone solution predictors.
Examples include learned multigrid and meta-multigrid schemes~\cite{greenfeld2019learning,chen2022meta}, neural-operator-accelerated Krylov methods~\cite{rudikov2024neural,kopanivcakova2025leveraging}, hybrids of neural operators and relaxation methods~\cite{zhang2024blending,lee2025hybrid,hujin2025hybrid,cui2025hybrid}, neural preconditioned Newton-type solvers~\cite{lee2025neural}, and preconditioners based on learned Green's functions~\cite{sun2026learning}.
Learned solvers have also been developed for Helmholtz equations~\cite{stanziola2021helmholtz,lee2025fast}, including multigrid-augmented deep preconditioners~\cite{azulay2022multigrid,lerer2024multigrid}, Wave-ADR-NS~\cite{cui2025neural}, and MGCFNN~\cite{xie2025mgcfnn}.
Training objectives have also been tailored to solver properties, using an $A$-energy error loss in FCG-NO~\cite{rudikov2024neural}, data-dependent losses for learned SPD preconditioners in PCG~\cite{li2023learning}, or scale-invariant losses for sparse approximate inverses in CG~\cite{yang2025learning}.
Other approaches combine condition and residual losses~\cite{li2025neural} or optimize residual--subspace angles through differentiable FGMRES~\cite{dimola2025neural}.
NPBS uses the same reference inverse to precondition the network input and the residual measured by the training loss.

\paragraph{Born-inspired neural solvers.}
The classical Born series/Lippmann--Schwinger formulations express wave scattering through a reference Green operator and a medium perturbation; CBS~\cite{osnabrugge2016convergent} modifies this construction to obtain a convergent fixed-point method for heterogeneous Helmholtz problems. Recent learned Born-series methods, such as Strong Scattering Neural Operator ($\mathrm{S}^{2}\mathrm{NO}$)~\cite{zeng2026ultrasoundtomographymusculoskeletaltissues} and Learned Born Series (LBS)~\cite{stanziola2024learned}, use Born-style unrolling to build forward wavefield simulators for scattering media. Their central use of the Born structure is therefore as a learned finite-depth surrogate for the wavefield map. NPBS instead uses Born-preconditioned residuals to iteratively update the solution, with a residual-based stopping criterion, and can also serve as a preconditioner in FGMRES.

\section{From Born Series to Left Preconditioning}
\label{sec:born_identity}
\subsection{Operator Splitting and Integral Equation}
We consider the Helmholtz equation for wave scattering in an unbounded heterogeneous medium, subject to the Sommerfeld radiation condition:
\begin{equation}\label{eq:helm}
\begin{aligned}
    \mathcal{L}u(x) := -\Delta u(x) - k(x)^2 u(x) &= f(x), \quad x \in \mathbb{R}^{d},\\
    \lim_{|x| \to \infty} |x|^{\frac{d-1}{2}} \left( \frac{\partial u}{\partial |x|} - i k_0 u \right) &= 0,
\end{aligned}
\end{equation}
where $u$ is the wavefield, $f$ the source, $d$ the spatial dimension, and $k(x)=\omega/c(x)$ the spatially varying wavenumber ($\omega$: angular frequency, $c(x)$: speed of sound). The constant $k_0$ is the background wavenumber ($k(x)\to k_0$ as $|x|\to\infty$).

For computation, the unbounded domain is truncated to a bounded domain $\Omega$, and absorbing boundary layers---sponge layers~\cite{israeli1981approximation} or PML~\cite{berenger1994perfectly,stanziola2021helmholtz}---are incorporated so that the resulting bounded-domain operator (still denoted $\mathcal{L}$) suppresses spurious reflections. After discretization (finite differences or FFT-based pseudospectral methods), we obtain the linear system
\begin{equation}
A \bm u = \bm f,
\end{equation}
where $A$ is the discrete Helmholtz operator (including absorbing-layer terms), $\bm u$ the discrete unknown, and $\bm f$ the discrete source.
Introduce a damped reference operator
\begin{equation}
\mathcal{L}_\eta := -\Delta - (k_0^2 + i\eta), \qquad \eta>0,
\end{equation}
used only for preconditioning/training coordinates. The reference is chosen to be transform-diagonalizable (periodic FFT or Dirichlet/Neumann DST/DCT) for fast application; any mismatch between this reference and the physical operator, including absorbing-layer terms, is absorbed into the perturbation
\begin{equation}
\mathcal{V}_\eta := \mathcal{L}_\eta - \mathcal{L}.
\end{equation}

Let $g_\eta$ be the Green kernel for the \emph{continuous} reference operator $\mathcal{L}_\eta$ under the chosen boundary condition on $\Omega$. Define the associated continuous Green operator
\begin{equation}
(\mathcal{G}_\eta q)(x) := \int_{\Omega} g_\eta(x,y)\,q(y)\,\mathrm{d}y,
\qquad q\in L^2(\Omega),
\label{eq:G_conv}
\end{equation}
so that $\mathcal{G}_\eta = \mathcal{L}_\eta^{-1}$ on the admissible function class.

\paragraph{Discrete implementation.}
In the discrete setting, let $L_\eta$ denote the matrix obtained by discretizing $\mathcal{L}_\eta$ (e.g., via finite differences or Fourier pseudospectral methods; see Appendix~\ref{app:fft_ps} for details), and define the discrete Green operator $G_\eta := L_\eta^{-1}$. For transform-diagonalizable references (periodic FFT or Dirichlet/Neumann DST/DCT), $G_\eta$ can be applied in $\mathcal{O}(N\log N)$ operations (Appendix~\ref{app:fft_verify_all}).

\paragraph{Operator splitting and Lippmann--Schwinger form.}
With $\mathcal{V}_\eta := \mathcal{L}_\eta - \mathcal{L}$, the heterogeneous Helmholtz operator splits as
\begin{equation}
\mathcal{L} = \mathcal{L}_\eta - \mathcal{V}_\eta.
\label{eq:A_split}
\end{equation}
For the coefficient-only Helmholtz split, $\mathcal{V}_\eta$ reduces to multiplication by $k(x)^2-k_0^2-i\eta$.

In the continuous setting, $\mathcal{L}u=f$ implies
\begin{equation}
\mathcal{L}_\eta u = \mathcal{V}_\eta u + f,
\qquad
u = \mathcal{G}_\eta(\mathcal{V}_\eta u + f),
\end{equation}
equivalently,
\begin{equation}
(\mathcal{I} - \mathcal{G}_\eta \mathcal{V}_\eta)u = \mathcal{G}_\eta f.
\end{equation}
In the discrete setting, starting from $A\bm u=\bm f$ and the discrete split $A=L_\eta - V_\eta$, we obtain
\begin{equation}
(I - G_\eta V_\eta)\bm u = G_\eta \bm f.
\label{eq:LS_discrete}
\end{equation}
Equation \eqref{eq:LS_discrete} is the preconditioned operator equation used by CBS~\cite{osnabrugge2016convergent}, whose fixed-point iteration is
\begin{equation}
\bm u^{m+1}
=
\bm u^m + \gamma\!\left(G_\eta(V_\eta \bm u^m + \bm f)-\bm u^m\right),
\label{eq:CBS}
\end{equation}
where $\gamma$ is a spatially varying diagonal preconditioner derived from the perturbation $V_\eta$.

\subsection{Equivalence to Shifted-Laplacian Left Preconditioning}
We state the algebraic equivalence that defines the preconditioned system underlying NPBS.

\begin{proposition}[Algebraic equivalence to left preconditioning]
\label{prop:equivalence_spectral}
Let the discrete Helmholtz operator admit the split $A=L_\eta - V_\eta$, and assume $L_\eta$ is invertible so that $G_\eta := L_\eta^{-1}$ is well-defined.
Then
\begin{equation}
I - G_\eta V_\eta = G_\eta A = L_\eta^{-1}A.
\label{eq:key_identity}
\end{equation}
Hence the Lippmann--Schwinger equation \eqref{eq:LS_discrete} is algebraically identical to the preconditioned system
\begin{equation}
L_\eta^{-1}A\,\bm u = L_\eta^{-1}\bm f.
\end{equation}
\end{proposition}

\begin{proof}
Since $G_\eta L_\eta=I$, we have $I-G_\eta V_\eta=G_\eta(L_\eta-V_\eta)=G_\eta A$.
\end{proof}

The identity \eqref{eq:key_identity} identifies $G_\eta\bm r$ as the residual of the left-preconditioned system, with the induced metric $R_\eta=G_\eta^\ast G_\eta$.

For a residual $\bm r$, the correction equation $A\delta\bm u=\bm r$ is equivalently $(G_\eta A)\delta\bm u=G_\eta\bm r$. When $A$ is nonsingular, this identifies $(G_\eta A)^{-1}$ as the ideal map from the Born-preconditioned residual to the solution correction.

\section{Neural Preconditioned Born Series Framework}
\label{sec:npbs}

A common learned linear-solver baseline uses a neural correction map $\mathcal{M}_\theta$ driven by the \emph{unpreconditioned} residual. (Since the learned map depends on the PDE coefficients $a$, it should formally be written $\mathcal{M}_\theta(\,\cdot\,;a)$; for brevity we write $\mathcal{M}_\theta(\,\cdot\,)$ below. The specific instantiation is given in Appendix~\ref{app:architectures}.)
Given the discrete system $A\bm u=\bm f$, define the residual $\bm r^m := \bm f - A\bm u^m$ and iterate
\begin{equation}
\bm u^{m+1} = \bm u^m + \mathcal{M}_\theta(\bm r^m)
= \bm u^m + \mathcal{M}_\theta(\bm f-A\bm u^m),
\label{eq:direct_iter}
\end{equation}
which operates on the unpreconditioned residual.

A standard training objective for $\mathcal{M}_\theta$ is to minimize the relative residual of the learned inverse action on random probes. Let $\mathcal{D}$ denote a probe distribution (e.g., $\mathcal{N}(0,I)$ or residual-replay samples). Then:
\begin{equation}
\mathcal{L}_{\mathrm{dir}}(\theta)
=
\E_{\bm r\sim\mathcal{D}}
\frac{\|A\,\mathcal{M}_\theta(\bm r)-\bm r\|_2}{\|\bm r\|_2}.
\label{eq:L_dir}
\end{equation}
This objective is simple and widely used~\cite{cui2025neural,xie2025mgcfnn}, but it optimizes the residual in \emph{unpreconditioned} Euclidean geometry; for indefinite Helmholtz operators, near-resonant modes can yield stiff gradients and slow learning.

\subsection{Neural Preconditioned Born Series (NPBS) iteration}
\label{sec:npbs_iter}

\paragraph{Residual-to-correction map.}
Let $\bm r^m=\bm f-A\bm u^m$ and $\bm r_{\mathrm{bs}}^m:=G_\eta\bm r^m$.
CBS and NPBS use the same Born-preconditioned residual, but map it to a solution correction differently:
\[
\delta\bm u_{\mathrm{CBS}}^m=\gamma(G_\eta\bm r^m),
\qquad
\delta\bm u_{\mathrm{NPBS}}^m=\mathcal M_\theta(G_\eta\bm r^m).
\]
NPBS retains the reference inverse $G_\eta$ and replaces the diagonal CBS correction $\gamma$ with a learned map. The network receives $\bm r_{\mathrm{bs}}^m$ and returns the correction added to $\bm u^m$:
\begin{equation}
\bm u^{m+1}=\bm u^m+\mathcal{M}_\theta\!\left(G_\eta(\bm f-A\bm u^m)\right).
\label{eq:npbs_iter}
\end{equation}
If $\mathcal M_\theta=\gamma$, this reduces to CBS (\eqref{eq:CBS}). If $A$ is nonsingular and $\mathcal M_\theta=(G_\eta A)^{-1}$, one step gives the exact solution.

\paragraph{Efficient integral-form implementation.}
\label{par:efficient_integral}
The identity \eqref{eq:key_identity} gives $G_\eta(\bm f-A\bm u^m)=G_\eta(V_\eta\bm u^m+\bm f)-\bm u^m$, so the same update can be evaluated as
\begin{equation}
\bm u^{m+1}
=
\bm u^m + \mathcal{M}_\theta\!\left(G_\eta(V_\eta \bm u^m + \bm f)-\bm u^m\right).
\label{eq:npbs_integral}
\end{equation}
Here $V_\eta\bm u^m$ is pointwise multiplication and applying $G_\eta$ requires one transform pair (e.g., FFT/IFFT).
Computing $A\bm u^m$ may require an additional application of the discrete Laplacian (e.g., an extra FFT in pseudospectral settings), hence we implement the efficient integral form \eqref{eq:npbs_integral} in practice.

\subsection{Metric-matched training within the NPBS framework}
\label{sec:bs_loss}

With the input $G_\eta\bm r$ fixed, we compare two ways to measure the residual after a learned correction. The first retains the Euclidean norm of the physical residual:
\begin{equation}
\mathcal{L}_{\mathrm{bs}}^{\ell_2}(\theta)
=
\E_{\bm r\sim\mathcal{D}}
\frac{
\left\lVert A\mathcal{M}_\theta(G_\eta \bm r)-\bm r\right\rVert_{2}
}{
\left\lVert \bm r\right\rVert_{2}
}.
\label{eq:L_bs_l2}
\end{equation}
For a physical residual $\bm r$, let $\bm z=G_\eta\bm r$. After the correction $\mathcal M_\theta(\bm z)$, the physical and Born-preconditioned residuals are
\[
\bm r^+=\bm r-A\mathcal M_\theta(\bm z),
\qquad
\bm z^+=G_\eta\bm r^+=\bm z-G_\eta A\mathcal M_\theta(\bm z).
\]
The per-sample term in \eqref{eq:L_bs_l2} is $\|\bm r^+\|_2/\|\bm r\|_2$, the relative physical residual after the correction. The next NPBS step receives $\bm z^+$. We therefore train the map to minimize the average ratio $\|\bm z^+\|_2/\|\bm z\|_2$. Using \eqref{eq:key_identity}, this gives the metric-matched objective:
\begin{equation}
\mathcal{L}_{\mathrm{bs}}^{R_\eta}(\theta)
=
\E_{\bm r\sim\mathcal{D}}
\frac{
\left\|
(I-G_\eta V_\eta)\,\mathcal{M}_\theta(G_\eta \bm r)-G_\eta \bm r
\right\|_2
}{
\|G_\eta \bm r\|_2
}.
\label{eq:L_bs_Reta_integral}
\end{equation}

\begin{proposition}[Equivalent weighted-residual formulation]
\label{prop:riesz_map}
To make the geometry underlying \eqref{eq:L_bs_Reta_integral} explicit, define the metric induced by the reference Green operator:
\begin{equation}
R_\eta:=G_\eta^\ast G_\eta,\qquad
\|\bm x\|_{R_\eta}:=\|G_\eta \bm x\|_2,\qquad
\langle \bm x,\bm y\rangle_{R_\eta}:=\langle \bm x,R_\eta \bm y\rangle_2.
\end{equation}
Then \eqref{eq:L_bs_Reta_integral} is equivalently
\begin{equation}
\mathcal{L}_{\mathrm{bs}}^{R_\eta}(\theta)
=
\E_{\bm r\sim\mathcal{D}}
\frac{
\left\lVert A\mathcal{M}_\theta(G_\eta \bm r)-\bm r\right\rVert_{R_\eta}
}{
\left\lVert \bm r\right\rVert_{R_\eta}
}.
\label{eq:L_bs_Reta_riesz}
\end{equation}
\end{proposition}

The proof is given in Appendix~\ref{app:proof_riesz_map}.

The losses $\mathcal{L}_{\mathrm{bs}}^{\ell_2}$ and $\mathcal{L}_{\mathrm{bs}}^{R_\eta}$ use the same input $G_\eta\bm r$ and the same physical residual $A\mathcal{M}_\theta(G_\eta\bm r)-\bm r$; only the norm differs. Thus, relative to \eqref{eq:L_dir}, the progression is: unpreconditioned probes with Euclidean norm ($\mathcal{L}_{\mathrm{dir}}$), then preconditioned probes with Euclidean norm ($\mathcal{L}_{\mathrm{bs}}^{\ell_2}$), and finally preconditioned probes with the $R_\eta$-metric ($\mathcal{L}_{\mathrm{bs}}^{R_\eta}$). A detailed Fourier-domain interpretation of the metric reweighting is given in Appendix~\ref{app:iterp_precond_geometry}.


A sufficient condition for convergence and a perturbation analysis are provided in Appendix~\ref{app:learned_convergence}.

The CDR and Newton derivations, transform implementations, and network details are provided in Appendices~\ref{app:cdr_nl_full}--\ref{app:architectures}.


\section{Experiments}
\label{sec:numerical_results}
We organize the evaluation around the heterogeneous Helmholtz equation. 
(i) we isolate the effects of iteration format and training loss across dataset difficulty and wavenumber regimes. 
(ii) we test whether the resulting gains persist across multiple neural-operator backbones. 
(iii) we compare stationary CBS and NPBS, and evaluate classical and learned preconditioners inside FGMRES.
(iv) we test the same iteration and loss design on three-dimensional Helmholtz problems.
(v) we evaluate on convection--diffusion--reaction (CDR) systems and Newton linear systems for nonlinear PDEs to demonstrate generality beyond Helmholtz. 
All methods stop when $\norm{\bm f-A\bm u^m}/\norm{\bm f}\leq\mathrm{rTol}$, with the tolerance specified for each experiment.
The primary learned preconditioner uses an FNO backbone (Appendix~\ref{app:architectures}); dataset splits, training hyperparameters, and compute resources are detailed in Appendices~\ref{app:experimental_setup} and~\ref{app:compute_resources}.

\subsection{Heterogeneous Helmholtz Equation}
\label{sec:exp_heter_helm_eq}
For the two-dimensional Helmholtz experiments, we select two datasets from OpenFWI~\cite{deng2022openfwi} with progressive difficulty: CurveVel-A (smooth curved velocity profiles) and CurveFault-B (sharp contrasts and fault structures); see Figure~\ref{fig:openfwi_dataset} for representative samples. Both use a $256\times256$ grid padded to $320\times320$ with 32-cell PML boundaries on each side, at $12$ points per wavelength ($\texttt{ppw}$). Full dataset and discretization details are given in Appendix~\ref{app:helmholtz_dataset}.

\subsubsection{Iteration format and loss comparison}
We isolate the effect of the network iteration format and the choice of training loss on solver efficiency. The evaluation compares three configurations under a unified stopping rule ($\text{rTol}=10^{-6}$): (i) Direct iteration trained with $\mathcal{L}_{\mathrm{dir}}$, (ii) NPBS iteration trained with $\mathcal{L}_{\mathrm{bs}}^{\ell_2}$, and (iii) NPBS iteration trained with the metric-matched objective $\mathcal{L}_{\mathrm{bs}}^{R_\eta}$.
Comparing (i) and (ii) isolates the change in residual input, while comparing (ii) and (iii) isolates the choice of residual metric.
\begin{table}[htbp]
\centering
\caption{Iteration format and loss comparison on OpenFWI datasets. Lower $\texttt{ppw}$ corresponds to higher wavenumber. Ratio: Direct+$\mathcal{L}_{\mathrm{dir}}$ iteration count divided by NPBS+$\mathcal{L}_{\mathrm{bs}}^{R_\eta}$ iteration count.}
\label{tab:loss_formulation}
\resizebox{\linewidth}{!}{%
\begin{tabular}{llccccccc}
\toprule
\multirow{2}{*}{Dataset} & \multirow{2}{*}{$\texttt{ppw}$} & \multicolumn{2}{c}{Direct + $\mathcal{L}_{\mathrm{dir}}$} & \multicolumn{2}{c}{NPBS + $\mathcal{L}_{\mathrm{bs}}^{\ell_2}$} & \multicolumn{2}{c}{NPBS + $\mathcal{L}_{\mathrm{bs}}^{R_\eta}$} & \multirow{2}{*}{Ratio} \\
\cmidrule(lr){3-4}\cmidrule(lr){5-6}\cmidrule(lr){7-8}
& & Avg. Iters & Avg. Time & Avg. Iters & Avg. Time & Avg. Iters & Avg. Time & \\
\midrule
CurveVel-A & 12 & 491 & 0.499\,s & 466 & 0.470\,s & \textbf{300} & \textbf{0.301\,s} & 1.6 \\
\midrule
\multirow{4}{*}{CurveFault-B}
& 24 &  531 & 0.537\,s &  517 & 0.537\,s & \textbf{458} & \textbf{0.403\,s} & 1.2 \\
& 20 &  616 & 0.642\,s &  539 & 0.550\,s & \textbf{475} & \textbf{0.465\,s} & 1.3 \\
& 16 &  895 & 0.914\,s &  842 & 0.848\,s & \textbf{538} & \textbf{0.544\,s} & 1.7 \\
& 12 & 1051 & 1.030\,s &  859 & 0.838\,s & \textbf{556} & \textbf{0.571\,s} & 1.9 \\
\bottomrule
\end{tabular}%
}
\end{table}

Table~\ref{tab:loss_formulation} reports the effect of iteration scheme and training loss across OpenFWI datasets and wavenumber regimes. 
NPBS~+~$\mathcal{L}_{\mathrm{bs}}^{R_\eta}$ consistently requires the fewest iterations and achieves the lowest runtime; its advantage over Direct~+~$\mathcal{L}_{\mathrm{dir}}$ grows both with stronger heterogeneity, from $1.6\times$ on CurveVel-A to $1.9\times$ on CurveFault-B, and with increasing wavenumber, from $1.2\times$ to $1.9\times$ as $\texttt{ppw}$ decreases from 24 to 12. 
These trends indicate that metric matching is most beneficial in harder Helmholtz regimes, while the runtime gains and the small $0.7\%$ training overhead (Appendix~\ref{app:training_convergence}) corroborate the efficient integral-form design in \S\ref{par:efficient_integral}.

\subsubsection{Cross-backbone robustness of iteration format and loss}
We compare all three formulations across five backbones: FNO~\cite{li2020fourier}, U-Net~\cite{ronneberger2015u}, MgNO~\cite{he2024mgno}, Encoder--solver based on~\cite{lerer2024multigrid}, and MGCFNN~\cite{xie2025mgcfnn}.
The FNO row reproduces the CurveFault-B, PPW $12$ results from Table~\ref{tab:loss_formulation}. Implementation details and baseline adaptations are described in Appendix~\ref{app:setup_helmholtz}.

Table~\ref{tab:learning_baselines} shows two robust trends. First, for every backbone, NPBS+$\mathcal{L}_{\mathrm{bs}}^{R_\eta}$ yields the lowest iteration count. Second, solve-time ranking differs from iteration ranking: although MGCFNN attains the fewest iterations, MgNO gives the lowest average solve time, reflecting its lightweight design and compile-friendly implementation.

\begin{table}[htbp]
\centering
\caption{Cross-backbone comparison on CurveFault-B with representative learning-based baselines.}
\label{tab:learning_baselines}
\begin{tabular}{lcccccc}
\toprule
\multirow{2}{*}{Backbone} & \multicolumn{2}{c}{Direct + $\mathcal{L}_{\mathrm{dir}}$} & \multicolumn{2}{c}{NPBS + $\mathcal{L}_{\mathrm{bs}}^{\ell_2}$} & \multicolumn{2}{c}{NPBS + $\mathcal{L}_{\mathrm{bs}}^{R_\eta}$} \\
\cmidrule(lr){2-3}\cmidrule(lr){4-5}\cmidrule(lr){6-7}
& Avg. Iters & Avg. Time & Avg. Iters & Avg. Time & Avg. Iters & Avg. Time \\
\midrule
FNO & 1051 & 1.030\,s & 859 & 0.838\,s & 556 & 0.571\,s \\
U-Net & 853 & 0.433\,s & 496 & 0.254\,s & 399 & 0.198\,s \\
MgNO & 342 & 0.268\,s & 141 & 0.123\,s & 123 & \textbf{0.107\,s} \\
Encoder--solver & 468 & 0.562\,s & 404 & 0.455\,s & 321 & 0.350\,s \\
MGCFNN & 69 & 0.336\,s & 66 & 0.334\,s & \textbf{48} & 0.235\,s \\
\bottomrule
\end{tabular}
\end{table}

\subsubsection{Comparison with classical and learned preconditioners}
Following the cross-backbone study, we select MgNO for its lowest average solve time among the evaluated neural architectures. On CurveFault-B, we compare stationary CBS~\cite{osnabrugge2016convergent} and NPBS, together with FGMRES~\cite{saad1993flexible} using CBS, shifted-Laplacian multigrid~\cite{erlangga2006novel} (SL), the reference inverse $G_\eta$, or learned MgNO corrections. All FGMRES runs use restart length~$50$.

Table~\ref{tab:helm_cbs_vs_npbs} shows that stationary NPBS reduces the average iteration count by $9.1\times$ relative to CBS, although its solve time is longer ($0.107$ vs.\ $0.086$\,s). FGMRES with metric-matched NPBS achieves the lowest reported solve time, with $37$ iterations and $0.078$\,s. FGMRES with directly trained MgNO requires $47$ iterations and $0.116$\,s. Using $G_\eta$ alone requires $894$ iterations and $2.821$\,s, so the reference inverse alone does not explain the learned preconditioner's performance.
Figure~\ref{fig:openfwi_infer} depicts a representative CurveFault-B sample with its wavefield solution and convergence profiles.

\begin{table}[htbp]
\centering
\caption{Classical and learned solvers on CurveFault-B. All learned corrections use MgNO.}
\label{tab:helm_cbs_vs_npbs}
\begin{tabular}{lcc}
\toprule
Method & Avg. Iters & Avg. Time \\
\midrule
CBS & 1124 & 0.086\,s \\
FGMRES+CBS & 547 & 1.820\,s \\
FGMRES+SL & 781 & 4.312\,s \\
FGMRES+$G_\eta$ & 894 & 2.821\,s \\
NPBS (MgNO) & 123 & 0.107\,s \\
FGMRES+MgNO & 47 & 0.116\,s \\
FGMRES+NPBS (MgNO) & \textbf{37} & \textbf{0.078}\,s \\
\bottomrule
\end{tabular}%
\end{table}

\begin{figure}[htbp]
    \centering
    \includegraphics[width=\linewidth]{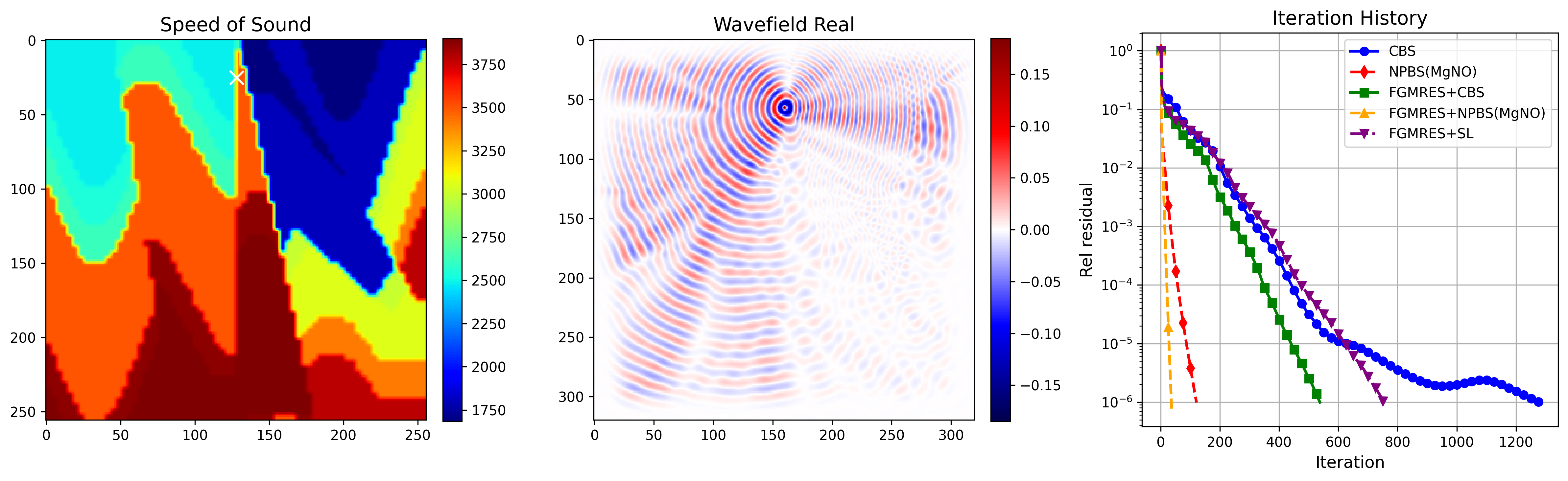}
    \caption{Iteration comparison of different strategies on a representative instance from the OpenFWI CurveFault-B dataset.}
    \label{fig:openfwi_infer}
\end{figure}

\subsubsection{Three-dimensional Helmholtz experiments}
\label{sec:exp_3d_helmholtz}
NPBS extends naturally to three dimensions with the same iteration and loss formulation. We evaluate the three formulations on OpenFWI Kimberlina-V1 using a three-dimensional MgNO backbone on $192^3$ grids at PPW $6$. Both NPBS variants improve over direct residual learning, with the metric-matched formulation achieving the fewest iterations and lowest solve time (Table~\ref{tab:helmholtz_3d}). Figure~\ref{fig:openfwi_3d_infer} illustrates a representative test case and its convergence histories. Experimental details are provided in Appendix~\ref{app:3d_setup}.

\begin{table}[htbp]
\centering
\caption{Three-dimensional Kimberlina-V1 Helmholtz results with MgNO.}
\label{tab:helmholtz_3d}
\begin{tabular}{lccc}
\toprule
 & Direct + $\mathcal L_{\mathrm{dir}}$ & NPBS + $\mathcal L_{\mathrm{bs}}^{\ell_2}$ & NPBS + $\mathcal L_{\mathrm{bs}}^{R_\eta}$ \\
\midrule
Avg.\ Iters & 445 & 68 & \textbf{41} \\
Avg.\ Time & 39.1\,s & 5.69\,s & \textbf{3.60\,s} \\
\bottomrule
\end{tabular}
\end{table}

\begin{figure}[!htbp]
    \centering
    \includegraphics[width=\linewidth]{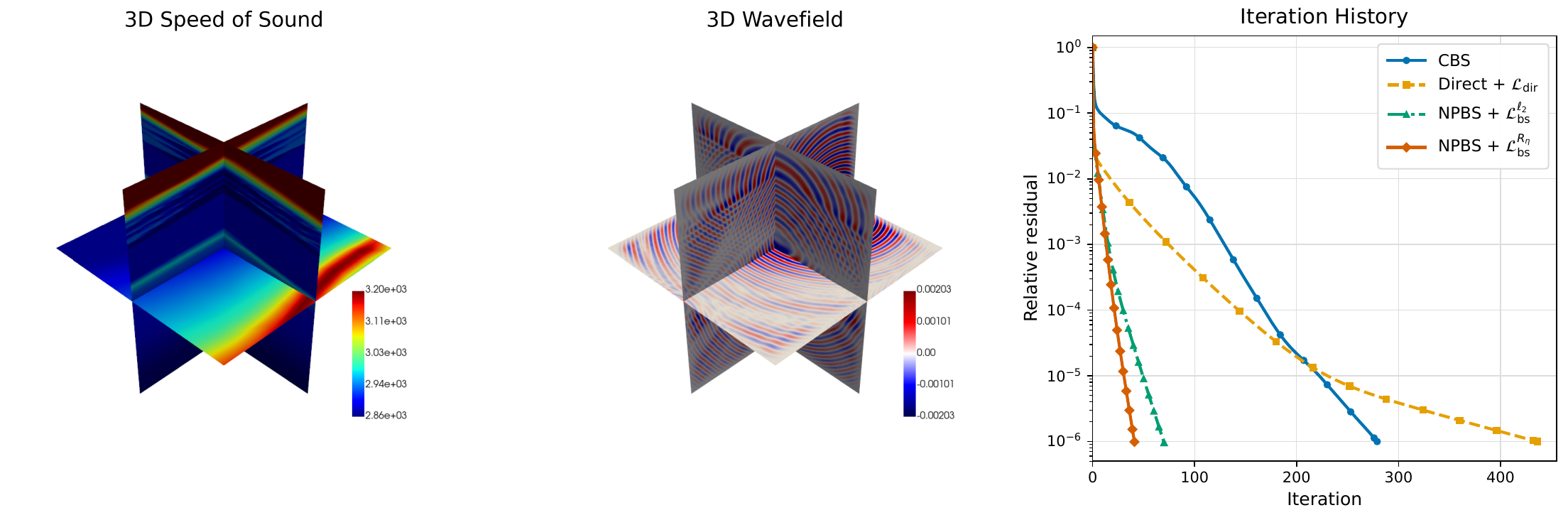}
    \caption{A three-dimensional Kimberlina-V1 test example. Left: slices of the velocity model. Middle: slices of the wavefield. Right: relative-residual histories for CBS and the three learned formulations.}
    \label{fig:openfwi_3d_infer}
\end{figure}

\FloatBarrier

\subsection{Extensions Beyond the Helmholtz Equation}
\label{sec:generalization}
We apply the same iteration format and loss design to a linear convection--diffusion--reaction (CDR) system and Newton linear systems for a nonlinear PDE, using an appropriate reference operator and a correction network trained separately for each PDE setting.
Dataset and setup details are in Appendix~\ref{app:setup_generalization} and \ref{app:dataset_figures}.

\paragraph{Linear Convection--Diffusion--Reaction (CDR) Systems}
We consider the linear CDR equation
\begin{equation}
  - \nabla \cdot (\kappa \nabla u) + \mathbf{v}\cdot \nabla u + \sigma u = f
  \quad \text{in } \Omega \subset \mathbb{R}^d,
\end{equation}
posed with periodic boundary conditions. Following the Helmholtz study, we compare three formulations and hold the architecture fixed so that only the iteration format and training objective vary.
We construct a synthetic dataset on a $128\times 128$ periodic grid with spatially heterogeneous coefficients (dataset generation details and representative samples are provided in Appendix~\ref{app:cdr_dataset}).

As shown in Table~\ref{tab:other_pde_solver_cmp} and Figure~\ref{fig:cdr_iter_his}, the advantage of NPBS with $\mathcal{L}_{\mathrm{bs}}^{R_\eta}$ observed on the Helmholtz equation carries over to the CDR setting: the metric-matched formulation consistently achieves the fastest convergence among the three objectives.

\begin{figure}[ht]
    \centering
    \includegraphics[width=0.85\linewidth]{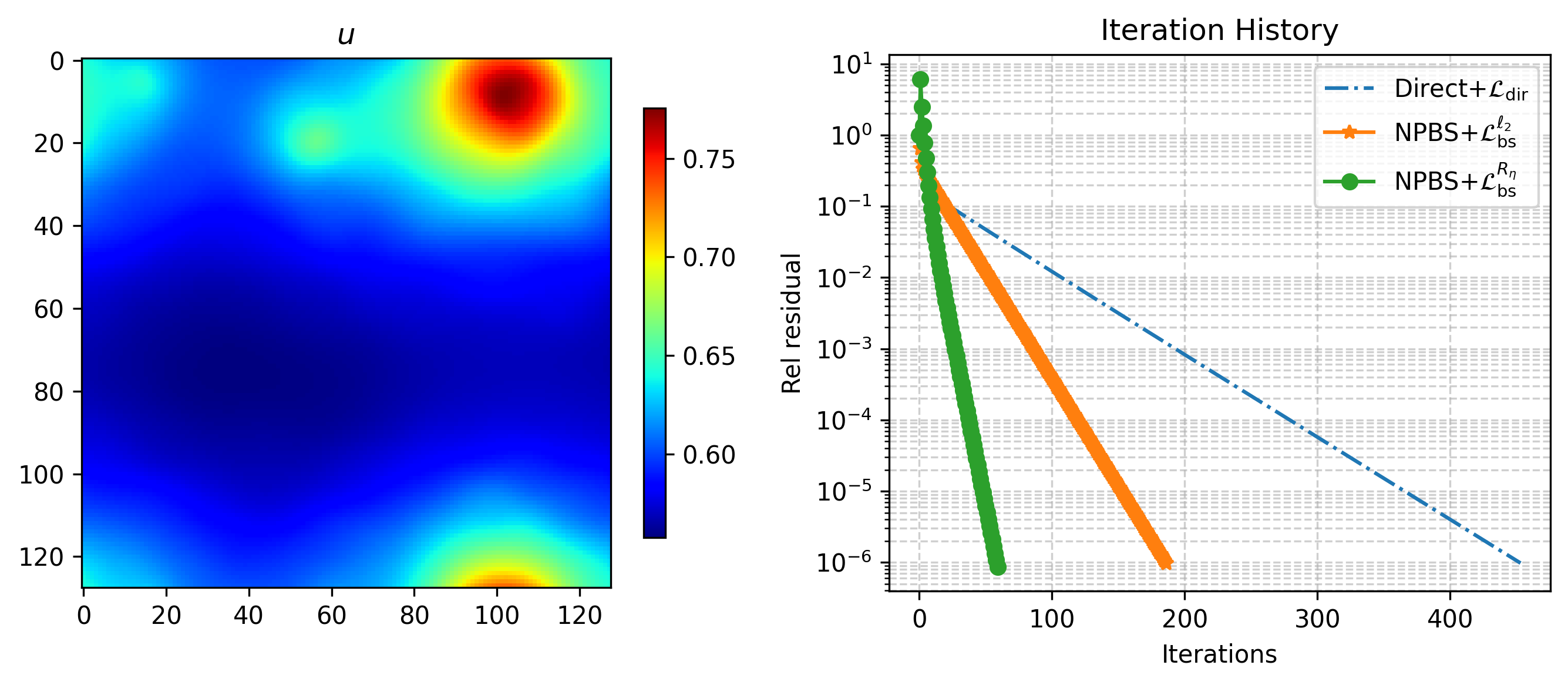}
    \caption{Representative CDR test case. Left: solution field $u$. Right: residual histories.}
    \label{fig:cdr_iter_his}
\end{figure}

\paragraph{Newton Linear Systems for Nonlinear PDEs}
We consider the non-convex problem with multiple solutions studied in~\cite{hao2024newton}:
\begin{equation}
\begin{cases}
  -\Delta u(x,y) - u(x,y)^2 = -s \sin(\pi x)\sin(\pi y), & (x,y) \in \Omega,\\[2mm]
  u(x,y) = 0, & (x,y) \in \partial\Omega,
\end{cases}
\end{equation}
with $f(x,y):=-s\sin(\pi x)\sin(\pi y)$ and define nonlinear residual $\mathcal{F}_{\mathrm{nl}}(u):=-\Delta u-u^2-f$. 
At Newton iterate $u^{(m)}$, the linearized correction $\delta u^{(m)}$ solves
\begin{equation}\label{eq:newton}
\begin{cases}
  -\Delta (\delta u^{(m)}) - 2u^{(m)}\,\delta u^{(m)} = -\mathcal{F}_{\mathrm{nl}}(u^{(m)}), & (x,y) \in \Omega,\\[1mm]
  \delta u^{(m)} = 0, & (x,y) \in \partial\Omega,
\end{cases}
\end{equation}
followed by the update $u^{(m+1)}=u^{(m)}+\delta u^{(m)}$. 
In every Newton step, the learned preconditioner is applied to the inner linear system~\eqref{eq:newton}.
Following~\cite{hao2024newton}, discretization and dataset details are in Appendix~\ref{app:newton_dataset}. For this Newton benchmark, the inner linear solves use $\mathrm{rTol}=10^{-4}$.

Table~\ref{tab:other_pde_solver_cmp} shows that the ranking observed in the linear benchmarks carries over to the Newton inner solves: the metric-matched NPBS requires the fewest iterations.
Figure~\ref{fig:newton_iter} further shows that the learned preconditioner integrates stably within the outer Newton loop.
\begin{figure}[htbp]
    \centering
    \includegraphics[width=0.95\linewidth]{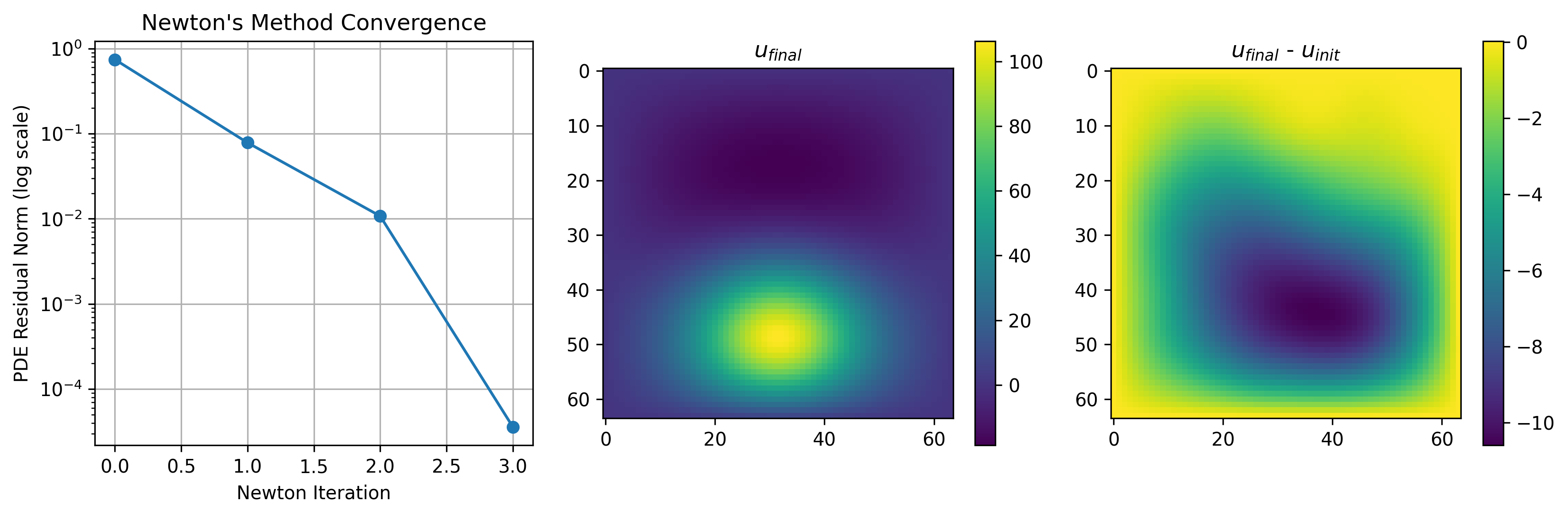}
    \caption{FNO preconditioner with NPBS+$\mathcal{L}_{\mathrm{bs}}^{R_\eta}$ embedded in Newton iterations.}
    \label{fig:newton_iter}
\end{figure}

\begin{table}[htbp]
\centering
\renewcommand{\arraystretch}{1.2}
\caption{Iteration format and loss comparison across other PDEs.}
\label{tab:other_pde_solver_cmp}
\begin{tabular}{l cc cc}
\toprule
\multirow{2}{*}{\textbf{Method}} & \multicolumn{2}{c}{\textbf{CDR Benchmark}} & \multicolumn{2}{c}{\textbf{Newton Inner-Solve}} \\
\cmidrule(lr){2-3} \cmidrule(lr){4-5}
& Avg.\ Iters & Avg.\ Rel Res & Avg.\ Iters & Avg.\ Rel Res \\
\midrule
Direct + $\mathcal{L}_{\mathrm{dir}}$          & 391          & $9.84\times10^{-7}$ & 25          & $8.98\times10^{-5}$ \\
NPBS + $\mathcal{L}_{\mathrm{bs}}^{\ell_2}$    & 167          & $9.57\times10^{-7}$ & 11          & $7.20\times10^{-5}$ \\
NPBS + $\mathcal{L}_{\mathrm{bs}}^{R_\eta}$    & \textbf{96}  & $8.68\times10^{-7}$ & \textbf{9}  & $6.75\times10^{-5}$ \\
\bottomrule
\end{tabular}
\end{table}

\section{Conclusion}
\label{sec:conclusion}
We introduced Neural Preconditioned Born Series (NPBS), a learned iterative preconditioning framework that operates in Born-preconditioned residual coordinates. The identity $I-G_\eta V_\eta = G_\eta A$ shows that these coordinates are exactly those of shifted-Laplacian left preconditioning, and the induced metric $R_\eta=G_\eta^\ast G_\eta$ leads to a metric-matched objective $\mathcal{L}_{\mathrm{bs}}^{R_\eta}$ that aligns training with inference.
Experiments on heterogeneous Helmholtz problems show that metric-matched NPBS consistently reduces iteration counts, with larger gains in more difficult regimes and robustness across neural backbones. FGMRES with NPBS gives the lowest solve time in the evaluated two-dimensional comparison. Results on CDR systems and Newton linear systems for nonlinear PDEs further indicate that residual-metric matching is useful beyond Helmholtz. Our current implementation is limited to structured Cartesian grids. Future work includes extending the reference inverse and correction networks to unstructured meshes, developing lighter neural preconditioners, and addressing time-dependent PDEs.

\begin{ack}
J.J. was supported by the focused project [FPF10120250014] of the Shenzhen Loop Area Institute and the National Natural Science Foundation of China [22341302]. X.L. was supported by the National Natural Science Foundation of China (Grant No. 12401571) and the Shandong Provincial Natural Science Foundation (Grant No. 2026HWYQ-024). The authors declare no competing interests.
\end{ack}

\bibliographystyle{plainnat}
\bibliography{references}

@article{saad1993flexible,
  title={A flexible inner-outer preconditioned GMRES algorithm},
  author={Saad, Youcef},
  journal={SIAM Journal on Scientific Computing},
  volume={14},
  number={2},
  pages={461--469},
  year={1993},
  publisher={SIAM}
}

@inproceedings{
he2024mgno,
title={Mg{NO}: Efficient Parameterization of Linear Operators via Multigrid},
author={Juncai He and Xinliang Liu and Jinchao Xu},
booktitle={The Twelfth International Conference on Learning Representations},
year={2024},
}

@article{kovachki2023neural,
  title={Neural operator: Learning maps between function spaces with applications to pdes},
  author={Kovachki, Nikola and Li, Zongyi and Liu, Burigede and Azizzadenesheli, Kamyar and Bhattacharya, Kaushik and Stuart, Andrew and Anandkumar, Anima},
  journal={Journal of Machine Learning Research},
  volume={24},
  number={89},
  pages={1--97},
  year={2023}
}

@article{osnabrugge2016convergent,
  title={A convergent Born series for solving the inhomogeneous Helmholtz equation in arbitrarily large media},
  author={Osnabrugge, Gerwin and Leedumrongwatthanakun, Saroch and Vellekoop, Ivo M},
  journal={Journal of computational physics},
  volume={322},
  pages={113--124},
  year={2016},
  publisher={Elsevier}
}

@article{erlangga2004class,
  title={On a class of preconditioners for solving the Helmholtz equation},
  author={Erlangga, Yogi A and Vuik, Cornelis and Oosterlee, Cornelis Willebrordus},
  journal={Applied Numerical Mathematics},
  volume={50},
  number={3-4},
  pages={409--425},
  year={2004},
  publisher={Elsevier}
}

@article{erlangga2006novel,
  title={A novel multigrid based preconditioner for heterogeneous Helmholtz problems},
  author={Erlangga, Yogi A and Oosterlee, Cornelis W and Vuik, Cornelis},
  journal={SIAM Journal on Scientific Computing},
  volume={27},
  number={4},
  pages={1471--1492},
  year={2006},
  publisher={SIAM}
}

@article{chen2013source,
  title={A source transfer domain decomposition method for Helmholtz equations in unbounded domain},
  author={Chen, Zhiming and Xiang, Xueshuang},
  journal={SIAM Journal on Numerical Analysis},
  volume={51},
  number={4},
  pages={2331--2356},
  year={2013},
  publisher={SIAM}
}

@article{ernst2011difficult,
  title={Why it is difficult to solve Helmholtz problems with classical iterative methods},
  author={Ernst, Oliver G and Gander, Martin J},
  journal={Numerical analysis of multiscale problems},
  pages={325--363},
  year={2011},
  publisher={Springer}
}

@article{gander2019class,
  title={A class of iterative solvers for the Helmholtz equation: Factorizations, sweeping preconditioners, source transfer, single layer potentials, polarized traces, and optimized Schwarz methods},
  author={Gander, Martin J and Zhang, Hui},
  journal={Siam Review},
  volume={61},
  number={1},
  pages={3--76},
  year={2019},
  publisher={SIAM}
}

@article{graham2020domain,
  title={Domain decomposition with local impedance conditions for the Helmholtz equation with absorption},
  author={Graham, Ivan G and Spence, Euan A and Zou, Jun},
  journal={SIAM Journal on Numerical Analysis},
  volume={58},
  number={5},
  pages={2515--2543},
  year={2020},
  publisher={SIAM}
}

@article{raissi2019physics,
  title={Physics-informed neural networks: A deep learning framework for solving forward and inverse problems involving nonlinear partial differential equations},
  author={Raissi, Maziar and Perdikaris, Paris and Karniadakis, George E},
  journal={Journal of Computational physics},
  volume={378},
  pages={686--707},
  year={2019},
  publisher={Elsevier}
}

@article{lu2021learning,
  title={Learning nonlinear operators via DeepONet based on the universal approximation theorem of operators},
  author={Lu, Lu and Jin, Pengzhan and Pang, Guofei and Zhang, Zhongqiang and Karniadakis, George Em},
  journal={Nature machine intelligence},
  volume={3},
  number={3},
  pages={218--229},
  year={2021},
  publisher={Nature Publishing Group UK London}
}

@article{li2020fourier,
  title={Fourier neural operator for parametric partial differential equations},
  author={Li, Zongyi and Kovachki, Nikola and Azizzadenesheli, Kamyar and Liu, Burigede and Bhattacharya, Kaushik and Stuart, Andrew and Anandkumar, Anima},
  journal={arXiv preprint arXiv:2010.08895},
  year={2020}
}

@article{li2023fourier,
  title={Fourier neural operator with learned deformations for pdes on general geometries},
  author={Li, Zongyi and Huang, Daniel Zhengyu and Liu, Burigede and Anandkumar, Anima},
  journal={Journal of Machine Learning Research},
  volume={24},
  number={388},
  pages={1--26},
  year={2023}
}

@article{cao2021choose,
  title={Choose a transformer: Fourier or galerkin},
  author={Cao, Shuhao},
  journal={Advances in neural information processing systems},
  volume={34},
  pages={24924--24940},
  year={2021}
}

@article{liu2024mitigating,
  title={Mitigating spectral bias for the multiscale operator learning},
  author={Liu, Xinliang and Xu, Bo and Cao, Shuhao and Zhang, Lei},
  journal={Journal of Computational Physics},
  volume={506},
  pages={112944},
  year={2024},
  publisher={Elsevier}
}

@inproceedings{rahaman2019spectral,
  title={On the spectral bias of neural networks},
  author={Rahaman, Nasim and Baratin, Aristide and Arpit, Devansh and Draxler, Felix and Lin, Min and Hamprecht, Fred and Bengio, Yoshua and Courville, Aaron},
  booktitle={International conference on machine learning},
  pages={5301--5310},
  year={2019},
  organization={PMLR}
}

@inproceedings{xie2025mgcfnn,
  title={MGCFNN: A Neural MultiGrid Solver with Novel Fourier Neural Network for High Wave Number Helmholtz Equations},
  author={Xie, Yan and Lv, Minrui and Zhang, Chen-Song},
  booktitle={The Thirteenth International Conference on Learning Representations},
  year={2025}
}

@inproceedings{greenfeld2019learning,
  title={Learning to optimize multigrid PDE solvers},
  author={Greenfeld, Daniel and Galun, Meirav and Basri, Ronen and Yavneh, Irad and Kimmel, Ron},
  booktitle={International Conference on Machine Learning},
  pages={2415--2423},
  year={2019},
  organization={PMLR}
}

@article{chen2022meta,
  title={Meta-mgnet: Meta multigrid networks for solving parameterized partial differential equations},
  author={Chen, Yuyan and Dong, Bin and Xu, Jinchao},
  journal={Journal of computational physics},
  volume={455},
  pages={110996},
  year={2022},
  publisher={Elsevier}
}

@article{azulay2022multigrid,
  title={Multigrid-augmented deep learning preconditioners for the Helmholtz equation},
  author={Azulay, Yael and Treister, Eran},
  journal={SIAM Journal on Scientific Computing},
  volume={45},
  number={3},
  pages={S127--S151},
  year={2022},
  publisher={SIAM}
}

@article{cui2025neural,
  title={A neural multigrid solver for helmholtz equations with high wavenumber and heterogeneous media},
  author={Cui, Chen and Jiang, Kai and Shu, Shi},
  journal={SIAM Journal on Scientific Computing},
  volume={47},
  number={3},
  pages={C655--C679},
  year={2025},
  publisher={SIAM}
}

@article{zhang2024blending,
  title={Blending neural operators and relaxation methods in PDE numerical solvers},
  author={Zhang, Enrui and Kahana, Adar and Kopani{\v{c}}{\'a}kov{\'a}, Alena and Turkel, Eli and Ranade, Rishikesh and Pathak, Jay and Karniadakis, George Em},
  journal={Nature Machine Intelligence},
  volume={6},
  number={11},
  pages={1303--1313},
  year={2024},
  publisher={Nature Publishing Group UK London}
}

@InProceedings{rudikov2024neural,
  title = 	 {Neural operators meet conjugate gradients: The {FCG}-{NO} method for efficient {PDE} solving},
  author =       {Rudikov, Alexander and Fanaskov, Vladimir and Muravleva, Ekaterina and Laevsky, Yuri M. and Oseledets, Ivan},
  booktitle = 	 {Proceedings of the 41st International Conference on Machine Learning},
  pages = 	 {42766--42782},
  year = 	 {2024},
  editor = 	 {Salakhutdinov, Ruslan and Kolter, Zico and Heller, Katherine and Weller, Adrian and Oliver, Nuria and Scarlett, Jonathan and Berkenkamp, Felix},
  volume = 	 {235},
  series = 	 {Proceedings of Machine Learning Research},
  month = 	 {21--27 Jul},
  publisher =    {PMLR},
}

@article{lerer2024multigrid,
  title={Multigrid-augmented deep learning preconditioners for the Helmholtz equation using compact implicit layers},
  author={Lerer, Bar and Ben-Yair, Ido and Treister, Eran},
  journal={SIAM Journal on Scientific Computing},
  volume={46},
  number={5},
  pages={S123--S144},
  year={2024},
  publisher={SIAM}
}

@article{berenger1994perfectly,
  title={A perfectly matched layer for the absorption of electromagnetic waves},
  author={Berenger, Jean-Pierre},
  journal={Journal of computational physics},
  volume={114},
  number={2},
  pages={185--200},
  year={1994},
  publisher={Elsevier}
}

@book{boyd2001chebyshev,
  title={Chebyshev and Fourier spectral methods},
  author={Boyd, John P},
  year={2001},
  publisher={Courier Corporation}
}

@article{israeli1981approximation,
  title={Approximation of radiation boundary conditions},
  author={Israeli, Moshe and Orszag, Steven A},
  journal={Journal of computational physics},
  volume={41},
  number={1},
  pages={115--135},
  year={1981},
  publisher={Elsevier}
}

@article{he2025self,
title = {Self-composing neural operators for high-frequency and multiscale PDE surrogates},
journal = {Journal of Computational Physics},
volume = {565},
pages = {115189},
year = {2026},
issn = {0021-9991},
doi = {https://doi.org/10.1016/j.jcp.2026.115189},
url = {https://www.sciencedirect.com/science/article/pii/S0021999126005413},
author = {Juncai He and Xinliang Liu and Jinchao Xu}
}

@article{babuska1997pollution,
  title={Is the pollution effect of the FEM avoidable for the Helmholtz equation considering high wave numbers?},
  author={Babuska, Ivo M and Sauter, Stefan A},
  journal={SIAM Journal on numerical analysis},
  volume={34},
  number={6},
  pages={2392--2423},
  year={1997},
  publisher={SIAM}
}

@article{stolk2016dispersion,
  title={A dispersion minimizing scheme for the 3-D Helmholtz equation based on ray theory},
  author={Stolk, Christiaan C},
  journal={Journal of computational Physics},
  volume={314},
  pages={618--646},
  year={2016},
  publisher={Elsevier}
}

@article{stanziola2021helmholtz,
  title={A Helmholtz equation solver using unsupervised learning: Application to transcranial ultrasound},
  author={Stanziola, Antonio and Arridge, Simon R and Cox, Ben T and Treeby, Bradley E},
  journal={Journal of computational physics},
  volume={441},
  pages={110430},
  year={2021},
  publisher={Elsevier}
}

@article{brandt1997wave,
  title={Wave-ray multigrid method for standing wave equations},
  author={Brandt, Achi and Livshits, Irene},
  journal={Electron. Trans. Numer. Anal},
  volume={6},
  number={162-181},
  pages={91},
  year={1997}
}

@book{trefethen2000spectral,
  title={Spectral methods in MATLAB},
  author={Trefethen, Lloyd N},
  year={2000},
  publisher={SIAM}
}

@book{shen2011spectral,
  title={Spectral methods: algorithms, analysis and applications},
  author={Shen, Jie and Tang, Tao and Wang, Li-Lian},
  volume={41},
  year={2011},
  publisher={Springer Science \& Business Media}
}

@article{hao2024newton,
  title={Newton informed neural operator for solving nonlinear partial differential equations},
  author={Hao, Wenrui and Liu, Xinliang and Yang, Yahong},
  journal={Advances in neural information processing systems},
  volume={37},
  pages={120832--120860},
  year={2024}
}

@article{deng2022openfwi,
  title={OpenFWI: Large-scale multi-structural benchmark datasets for full waveform inversion},
  author={Deng, Chengyuan and Feng, Shihang and Wang, Hanchen and Zhang, Xitong and Jin, Peng and Feng, Yinan and Zeng, Qili and Chen, Yinpeng and Lin, Youzuo},
  journal={Advances in Neural Information Processing Systems},
  volume={35},
  pages={6007--6020},
  year={2022}
}

@article{stanziola2024learned,
  title={A learned Born series for highly-scattering media},
  author={Stanziola, Antonio and Arridge, Simon and Cox, Ben T and Treeby, Bradley E},
  journal={JASA Express Letters},
  volume={3},
  number={5},
  year={2023},
  publisher={AIP Publishing}
}

@inproceedings{ronneberger2015u,
  title={U-net: Convolutional networks for biomedical image segmentation},
  author={Ronneberger, Olaf and Fischer, Philipp and Brox, Thomas},
  booktitle={International Conference on Medical image computing and computer-assisted intervention},
  pages={234--241},
  year={2015},
  organization={Springer}
}

@article{huang2021finite,
  title={A finite-difference iterative solver of the Helmholtz equation for frequency-domain seismic wave modeling and full-waveform inversion},
  author={Huang, Xingguo and Greenhalgh, Stewart},
  journal={Geophysics},
  volume={86},
  number={2},
  pages={T107--T116},
  year={2021},
  publisher={Society of Exploration Geophysicists}
}

@article{sun2020surrogate,
  title={Surrogate modeling for fluid flows based on physics-constrained deep learning without simulation data},
  author={Sun, Luning and Gao, Han and Pan, Shaowu and Wang, Jian-Xun},
  journal={Computer Methods in Applied Mechanics and Engineering},
  volume={361},
  pages={112732},
  year={2020},
  publisher={Elsevier}
}

@article{MELCHERS2026118893,
title = {Neural Green's operators for parametric partial differential equations},
journal = {Computer Methods in Applied Mechanics and Engineering},
volume = {455},
pages = {118893},
year = {2026},
doi = {https://doi.org/10.1016/j.cma.2026.118893},
author = {H.A. Melchers and J.H.M. Prins and M.R.A. Abdelmalik},
}

@article{kopanivcakova2025leveraging,
  title={Leveraging operator learning to accelerate convergence of the preconditioned conjugate gradient method},
  author={Kopani{\v{c}}{\'a}kov{\'a}, Alena and Lee, Youngkyu and Karniadakis, George Em},
  journal={Machine Learning for Computational Science and Engineering},
  volume={1},
  number={2},
  pages={39},
  year={2025},
  publisher={Springer}
}

@article{lee2025fast,
  title={Fast meta-solvers for 3D complex-shape scatterers using neural operators trained on a non-scattering problem},
  author={Lee, Youngkyu and Liu, Shanqing and Zou, Zongren and Kahana, Adar and Turkel, Eli and Ranade, Rishikesh and Pathak, Jay and Karniadakis, George Em},
  journal={Computer Methods in Applied Mechanics and Engineering},
  volume={446},
  pages={118231},
  year={2025},
  publisher={Elsevier}
}

@article{lee2025neural,
  title={A Neural-Operator Preconditioned Newton Method for Accelerated Nonlinear Solvers},
  author={Lee, Youngkyu and Liu, Shanqing and Darbon, Jerome and Karniadakis, George Em},
  journal={arXiv preprint arXiv:2511.08811},
  year={2025}
}

@article{lee2025hybrid,
  title={Hybrid Iterative Solvers with Geometry-Aware Neural Preconditioners for Parametric PDEs},
  author={Lee, Youngkyu and Florencio, Francesc Levrero and Pathak, Jay and Karniadakis, George Em},
  journal={arXiv preprint arXiv:2512.14596},
  year={2025}
}

@inproceedings{li2023learning,
  title={Learning preconditioners for conjugate gradient PDE solvers},
  author={Li, Yichen and Chen, Peter Yichen and Du, Tao and Matusik, Wojciech},
  booktitle={International Conference on Machine Learning},
  pages={19425--19439},
  year={2023},
  organization={PMLR}
}

@article{li2025neural,
  title={Neural preconditioning operator for efficient pde solves},
  author={Li, Zhihao and Xiao, Di and Lai, Zhilu and Wang, Wei},
  journal={arXiv preprint arXiv:2502.01337},
  year={2025}
}

@article{dimola2025neural,
  title={Neural preconditioning via Krylov subspace geometry},
  author={Dimola, Nunzio and Coclite, Alessandro and Zunino, Paolo},
  journal={Bollettino dell'Unione Matematica Italiana},
  pages={1--23},
  year={2025},
  publisher={Springer}
}

@article{zeng2025openbreastus,
  title={Openbreastus: Benchmarking neural operators for wave imaging using breast ultrasound computed tomography},
  author={Zeng, Zhijun and Zheng, Youjia and Hu, Hao and Dong, Zeyuan and Zheng, Yihang and Liu, Xinliang and Wang, Jinzhuo and Shi, Zuoqiang and Zhang, Linfeng and Li, Yubing and others},
  journal={arXiv preprint arXiv:2507.15035},
  year={2025}
}

@inproceedings{cao2025spectral,
  title={Spectral-refiner: Accurate fine-tuning of spatiotemporal fourier neural operator for turbulent flows},
  author={Cao, Shuhao and Brarda, Francesco and Li, Ruipeng and Xi, Yuanzhe},
  booktitle={International Conference on Learning Representations},
  volume={2025},
  pages={60348--60384},
  year={2025}
}

@article{hujin2025hybrid,
  title={A hybrid iterative method based on MIONet for PDEs: Theory and numerical examples},
  author={Hu, Jun and Jin, Pengzhan},
  journal={Mathematics of Computation},
  year={2025}
}

@article{sun2026learning,
title = {Learning singularity-encoded Green’s functions with application to iterative methods},
journal = {Journal of Computational Physics},
volume = {559},
pages = {114894},
year = {2026},
issn = {0021-9991},
doi = {https://doi.org/10.1016/j.jcp.2026.114894},
author = {Qi Sun and Shengyan Li and Bowen Zheng and Lili Ju and Xuejun Xu},
}

@inproceedings{
yang2025learning,
title={Learning Sparse Approximate Inverse Preconditioners for Conjugate Gradient Solvers on {GPU}s},
author={Zherui Yang and Zhehao Li and Kangbo Lyu and Yixuan Li and Tao Du and Ligang Liu},
booktitle={The Thirty-ninth Annual Conference on Neural Information Processing Systems},
year={2025},
url={https://openreview.net/forum?id=jtMDzggo6M}
}

@article{cui2025hybrid,
  title={A hybrid iterative neural solver based on spectral analysis for parametric PDEs},
  author={Cui, Chen and Jiang, Kai and Liu, Yun and Shu, Shi},
  journal={Journal of Computational Physics},
  volume={538},
  pages={114165},
  year={2025},
  publisher={Elsevier}
}

@misc{zeng2026ultrasoundtomographymusculoskeletaltissues,
      title={Ultrasound Tomography of Musculoskeletal Tissues with Generative Neural Physics},
      author={Zhijun Zeng and Youjia Zheng and Chang Su and Qianhang Wu and Hao Hu and Zeyuan Dong and Yang Lv and Ligang Cui and Zhiyong Hou and Weijun Lin and Zuoqiang Shi and Yubing Li and He Sun},
      year={2026},
      eprint={2508.12226},
      archivePrefix={arXiv},
      primaryClass={cs.CV},
      url={https://arxiv.org/abs/2508.12226},
}

\newpage
\appendix

\section{Proofs and Supporting Analysis}
\label{app:proofs}

\subsection{Proof of Proposition~\ref{prop:riesz_map}}
\label{app:proof_riesz_map}

\begin{proof}
Using \eqref{eq:key_identity}, \eqref{eq:L_bs_Reta_integral} can be rewritten as
\begin{equation}
\mathcal{L}_{\mathrm{bs}}^{R_\eta}(\theta)
=
\E_{\bm r\sim\mathcal{D}}
\frac{
\left\|
G_\eta A\,\mathcal{M}_\theta(G_\eta \bm r)-G_\eta \bm r
\right\|_2
}{
\|G_\eta \bm r\|_2
}.
\label{eq:L_bs_Reta_GA}
\end{equation}
By definition, $\|\bm x\|_{R_\eta}=\|G_\eta \bm x\|_2$. Substituting
$\bm x = A\mathcal{M}_\theta(G_\eta \bm r)-\bm r$ gives
$\|A\mathcal{M}_\theta(G_\eta \bm r)-\bm r\|_{R_\eta}
=
\|G_\eta A\mathcal{M}_\theta(G_\eta \bm r)-G_\eta \bm r\|_2$,
and similarly $\|\bm r\|_{R_\eta}=\|G_\eta \bm r\|_2$, yielding \eqref{eq:L_bs_Reta_riesz}.
\end{proof}

\subsection{Interpretation in preconditioned geometry.}
\label{app:iterp_precond_geometry}
Relative to \eqref{eq:L_dir}, \eqref{eq:L_bs_l2} and \eqref{eq:L_bs_Reta_integral} replace the direct regression target $A^{-1}$ with the preconditioned target $(G_\eta A)^{-1}$ by feeding $G_\eta \bm r$ into $\mathcal{M}_\theta$. The key difference lies in the residual metric: $\mathcal{L}_{\mathrm{bs}}^{\ell_2}$ measures the physical residual $e(\bm r;\theta):=A\,\mathcal{M}_\theta(G_\eta \bm r)-\bm r$ in the Euclidean norm $\|e(\bm r;\theta)\|_2$, whereas $\mathcal{L}_{\mathrm{bs}}^{R_\eta}$ measures it in the preconditioned norm $\|e(\bm r;\theta)\|_{R_\eta}=\|G_\eta e(\bm r;\theta)\|_2$. By Proposition~\ref{prop:riesz_map}, the latter is exactly the residual ratio evaluated in the same $R_\eta$-geometry induced by the preconditioned coordinates used by NPBS inference. For the constant-coefficient Helmholtz reference, the Fourier symbol
\begin{equation}
\widehat{G_\eta}(\bm \xi)=\frac{1}{|\bm \xi|^2-k_0^2-i\eta}
\end{equation}
shows how the metric reweights residual modes through $G_\eta$.

\subsection{Conditional convergence and bounded perturbations}
\label{app:learned_convergence}
Fix $A$ and $L_\eta$, with $L_\eta$ invertible and $G_\eta=L_\eta^{-1}$, and write $B_\eta=G_\eta A$ and $\bm z^m=G_\eta(\bm f-A\bm u^m)$. Equation~\eqref{eq:npbs_iter} gives
\begin{equation}
\bm z^{m+1}=\bm z^m-B_\eta\mathcal M_\theta(\bm z^m).
\label{eq:learned_residual_recursion}
\end{equation}
If $B_\eta$ is nonsingular and a uniform $0\le\rho<1$ satisfies
\begin{equation}
\|\bm z-B_\eta\mathcal M_\theta(\bm z)\|_2\le\rho\|\bm z\|_2
\label{eq:learned_reduction}
\end{equation}
at every residual encountered, then
\begin{equation}
\begin{aligned}
\|\bm z^m\|_2 &\le \rho^m\|\bm z^0\|_2,\\
\|\bm u^m-\bm u^\star\|_2 &\le \|B_\eta^{-1}\|_2\,\rho^m\|\bm z^0\|_2.
\end{aligned}
\label{eq:learned_convergence_bounds}
\end{equation}
where $A\bm u^\star=\bm f$. Thus, uniform reduction of the Born-preconditioned residual implies geometric convergence to the solution. The loss \eqref{eq:L_bs_Reta_integral} averages the corresponding one-step residual ratio.

\begin{proof}
Substitution of \eqref{eq:npbs_iter} into $\bm z^{m+1}=G_\eta(\bm f-A\bm u^{m+1})$ yields \eqref{eq:learned_residual_recursion}.
Under \eqref{eq:learned_reduction}, induction gives $\|\bm z^m\|_2\le\rho^m\|\bm z^0\|_2$.
If $B_\eta$ is nonsingular, $\bm z^m=B_\eta(\bm u^\star-\bm u^m)$, which proves the solution-error bound in \eqref{eq:learned_convergence_bounds}.
Also, $\|\bm f-A\bm u^m\|_2\le\|L_\eta\|_2\rho^m\|\bm z^0\|_2$, so the physical residual converges for this fixed system.
\end{proof}

\paragraph{Bounded perturbations.}
For updates with additive perturbations,
\begin{equation}
\bm z^{m+1}=\bm z^m-B_\eta\mathcal M_\theta(\bm z^m)+\bm d^m,
\qquad \|\bm d^m\|_2\le\varepsilon,
\end{equation}
assume that \eqref{eq:learned_reduction} holds at all encountered perturbed states with the same $\rho<1$.
The triangle inequality gives $\|\bm z^{m+1}\|_2\le\rho\|\bm z^m\|_2+\varepsilon$, and hence
\begin{equation}
\|\bm z^m\|_2\le\rho^m\|\bm z^0\|_2+\frac{1-\rho^m}{1-\rho}\varepsilon.
\end{equation}
This is a conditional bound on accumulated residual perturbations. Minimizing the expected one-step residual ratio in \eqref{eq:L_bs_Reta_integral} does not establish the uniform condition \eqref{eq:learned_reduction} on inference trajectories or guarantee stability on unseen problems. The finite-test contraction statistics in Appendix~\ref{app:rebuttal_diagnostics} are empirical evidence rather than a global guarantee.

\subsection{Conditional spectral bounds}
\label{app:conditional_spectral_bounds}
Under the assumptions of Proposition~\ref{prop:equivalence_spectral}, if additionally $\|G_\eta V_\eta\|_2\le q<1$, then
\begin{equation}
\sigma(I-G_\eta V_\eta)\subset\{z\in\mathbb{C}:|z-1|\le q\},
\label{eq:disk_bound}
\end{equation}
and
\begin{equation}
\kappa_2(I-G_\eta V_\eta)\le \frac{1+q}{1-q}.
\label{eq:kappa_bound}
\end{equation}

\begin{proof}
Since $\rho(G_\eta V_\eta)\le\|G_\eta V_\eta\|_2\le q$, every eigenvalue $\mu$ of $G_\eta V_\eta$ satisfies $|\mu|\le q$. Eigenvalues of $I-G_\eta V_\eta$ are $1-\mu$, giving \eqref{eq:disk_bound}. The condition $\|G_\eta V_\eta\|_2<1$ ensures the Neumann series converges:
\begin{equation}
(I-G_\eta V_\eta)^{-1}=\sum_{j=0}^{\infty}(G_\eta V_\eta)^j,
\label{eq:neumann_bound}
\end{equation}
so that
\[
\|I-G_\eta V_\eta\|_2\le 1+q,\qquad
\|(I-G_\eta V_\eta)^{-1}\|_2\le \sum_{j\ge 0}q^j=\frac{1}{1-q},
\]
which yields \eqref{eq:kappa_bound}.
\end{proof}
See also~\cite{huang2021finite} for related analysis.
The identity \eqref{eq:key_identity} and the metric $R_\eta=G_\eta^\ast G_\eta$ require only an invertible $L_\eta$; the bounds \eqref{eq:disk_bound}--\eqref{eq:kappa_bound} require the additional norm condition. The norm estimates in Appendix~\ref{app:spectral_diagnostics} do not support invoking these bounds to explain the observed performance. Exceeding one removes this spectral guarantee; it does not by itself imply divergence of the learned iteration.

\subsection{Role of \texorpdfstring{$\eta$}{eta} in Metric and Iteration Coordinates}
\label{app:role_eta}

The shift $\eta$ defines both the left preconditioner $G_\eta:=L_\eta^{-1}$ and the induced residual metric
$R_\eta:=G_\eta^\ast G_\eta$.
Consequently, the $R_\eta$-weighted residual norm satisfies
\begin{equation}
\|\bm f-A\bm u\|_{R_\eta}^2 := (\bm f-A\bm u)^\ast R_\eta (\bm f-A\bm u)=\|G_\eta(\bm f-A\bm u)\|_2^2,
\end{equation}
so $\eta$ controls the geometry in which residuals are measured and corrected, not only numerical damping.
In practice, we tune $\eta$ via a small validation sweep and then keep it fixed between training and inference
for each dataset to preserve objective--inference consistency.
Heuristically, increasing $\eta$ can move the reference inverse farther from resonance and may improve the contraction behavior of the fixed-point map, whereas excessively large damping can also weaken the usefulness of the resulting correction directions.
The sensitivity experiment in Appendix~\ref{app:eta_sensitivity} quantifies this trade-off around the canonical CBS scale.

\section{Full Derivations for CDR and Nonlinear Extensions}
\label{app:cdr_nl_full}
This appendix gives the full technical derivations deferred from the main text.

\subsection{Convection-Diffusion-Reaction (Linear)}
Consider
\begin{equation}
\mathcal{L}_{\mathrm{cdr}}u
=
-\nabla\cdot(\kappa(x)\nabla u)
\;+\;
\mathbf{v}(x)\cdot\nabla u
\;+\;
\sigma(x)u
=
f.
\label{eq:cdr_pde}
\end{equation}
Choose background constants $\kappa_0,\mathbf{v}_0,\sigma_0$ and define
\begin{equation}
\kappa=\kappa_0+\delta\kappa,\qquad
\mathbf{v}=\mathbf{v}_0+\delta\mathbf{v},\qquad
\sigma=\sigma_0+\delta\sigma.
\end{equation}
Set
\begin{equation}
\mathcal{L}_{0,\mathrm{cdr}}u
:=
-\kappa_0\Delta u
\;+\;
\mathbf{v}_0\cdot\nabla u
\;+\;
\sigma_0 u,
\label{eq:cdr_L0}
\end{equation}
and define $\mathcal{V}_{\mathrm{cdr}}$ by $\mathcal{L}_{\mathrm{cdr}}=\mathcal{L}_{0,\mathrm{cdr}}-\mathcal{V}_{\mathrm{cdr}}$, i.e.,
\begin{equation}
\mathcal{V}_{\mathrm{cdr}}u
=
\nabla\cdot(\delta\kappa\nabla u)
\;-\;
\delta\mathbf{v}\cdot\nabla u
\;-\;
\delta\sigma\,u.
\label{eq:cdr_V}
\end{equation}
Then
\begin{equation}
\mathcal{L}_{0,\mathrm{cdr}}u
=
\mathcal{V}_{\mathrm{cdr}}u+f.
\end{equation}
With $G_{\mathrm{cdr}}:=\mathcal{L}_{0,\mathrm{cdr}}^{-1}$:
\begin{equation}
(I-G_{\mathrm{cdr}}\mathcal{V}_{\mathrm{cdr}})u
=
G_{\mathrm{cdr}}f
=
G_{\mathrm{cdr}}\mathcal{L}_{\mathrm{cdr}}u.
\label{eq:cdr_integral}
\end{equation}
In discrete form ($A_{\mathrm{cdr}},V_{\mathrm{cdr}},G_{\mathrm{cdr}}$):
\begin{equation}
I-G_{\mathrm{cdr}}V_{\mathrm{cdr}}
=
G_{\mathrm{cdr}}A_{\mathrm{cdr}}.
\label{eq:cdr_identity}
\end{equation}
Therefore, the CDR-specific instance of the metric-matched Born Series inspired objective is
\begin{equation}
\mathcal{L}_{\mathrm{bs}}^{R_{\mathrm{cdr}}}(\theta)
=
\E_{r\sim\mathcal{D}}
\frac{
\norm{
(I-G_{\mathrm{cdr}}V_{\mathrm{cdr}})
\mathcal{M}_\theta(G_{\mathrm{cdr}}r)
-G_{\mathrm{cdr}}r
}
}{
\norm{G_{\mathrm{cdr}}r}
}
\label{eq:cdr_bs_loss}
\end{equation}
or equivalently
\begin{equation}
\mathcal{L}_{\mathrm{bs}}^{R_{\mathrm{cdr}}}(\theta)
=
\E_{r\sim\mathcal{D}}
\frac{
\norm{
G_{\mathrm{cdr}}A_{\mathrm{cdr}}
\mathcal{M}_\theta(G_{\mathrm{cdr}}r)
-G_{\mathrm{cdr}}r
}
}{
\norm{G_{\mathrm{cdr}}r}
}.
\label{eq:cdr_bs_loss_eq}
\end{equation}
with $R_{\mathrm{cdr}}:=G_{\mathrm{cdr}}^\ast G_{\mathrm{cdr}}$. This is exactly the same template as $\mathcal{L}_{\mathrm{bs}}^{R_\eta}$ in Section~\ref{sec:bs_loss}, under the substitution $(G_\eta,A,V_\eta)\mapsto(G_{\mathrm{cdr}},A_{\mathrm{cdr}},V_{\mathrm{cdr}})$.

For periodic domains, $G_{\mathrm{cdr}}$ is a convolution operator with Fourier symbol
\begin{equation}
\widehat{G}_{\mathrm{cdr}}(\xi)
=
\frac{1}{\kappa_0|\xi|^2+i\mathbf{v}_0\cdot\xi+\sigma_0},
\label{eq:cdr_G_symbol}
\end{equation}
so high-frequency diffusion-dominated modes are attenuated in training. As in Helmholtz, the objective measures the residual in the geometry of the preconditioned system.

\subsection{Nonlinear PDE via Newton-Born Linearization}
For the nonlinear example
\begin{equation}
\mathcal{F}_{\mathrm{nl}}(u):=-\Delta u-u^2-f=0,
\label{eq:nonlinear_residual}
\end{equation}
Newton at iterate $u^{(m)}$ solves
\begin{equation}
J_m\,\delta u^{(m)} = r^{(m)},
\qquad
J_m:=-\Delta-2u^{(m)},
\qquad
r^{(m)}:=-\mathcal{F}_{\mathrm{nl}}(u^{(m)}),
\label{eq:newton_linear_system}
\end{equation}
followed by $u^{(m+1)}=u^{(m)}+\delta u^{(m)}$.

\paragraph{Discrete realization used in this work (Dirichlet + FDM + DST).}
For the Newton experiments in Section~\ref{sec:numerical_results}, we use homogeneous Dirichlet boundary conditions on $\Omega=[0,1]^2$ and a five-point finite-difference discretization on interior unknowns. With $h=1/(N+1)$, define the discrete negative Laplacian
\begin{equation}
L_{\mathrm{D}}u
=
\frac{1}{h^2}
\begin{bmatrix}
0 & -1 & 0\\
-1 & 4 & -1\\
0 & -1 & 0
\end{bmatrix}*u,
\end{equation}
so the discrete Jacobian is
\begin{equation}
J_m^{\mathrm{disc}} = L_{\mathrm{D}} - 2\,\mathrm{diag}(u^{(m)}).
\end{equation}
In the following discrete formulation, we write $J_m=J_m^{\mathrm{disc}}$ and use the discrete Dirichlet Poisson operator $L_{\mathrm{D}}$ as the fixed reference, with inverse
\begin{equation}
G:=L_{\mathrm{D}}^{-1}.
\label{eq:newton_ref}
\end{equation}
The Jacobian admits the split $J_m=L_{\mathrm{D}}-V_m$, where
\begin{equation}
V_m:=L_{\mathrm{D}}-J_m=2\,\mathrm{diag}(u^{(m)}).
\label{eq:newton_vm}
\end{equation}
The fixed reference inverse $G$ is applied using the DST-I. The identity $I-GV_m=GJ_m$ gives the preconditioned correction equation
\begin{equation}
(I-GV_m)\,\delta u^{(m)}=G r^{(m)}.
\label{eq:newton_born_identity}
\end{equation}
Hence the Newton-state instance of the metric-matched Born Series objective is
\begin{equation}
\mathcal{L}_{\mathrm{bs}}^{R}(\theta)
=
\E_{(u^{(m)},r)\sim\mathcal{D}_{\mathrm{N}}}
\frac{
\norm{
(I-GV_m)\mathcal{M}_\theta(G r)-G r
}
}{
\norm{G r}
},
\label{eq:nl_bs_loss}
\end{equation}
equivalently
\begin{equation}
\mathcal{L}_{\mathrm{bs}}^{R}(\theta)
=
\E_{(u^{(m)},r)\sim\mathcal{D}_{\mathrm{N}}}
\frac{
\norm{
GJ_m\mathcal{M}_\theta(G r)-G r
}
}{
\norm{G r}
}.
\label{eq:nl_bs_loss_eq}
\end{equation}
Here $R:=G^\ast G$ is fixed across Newton steps, while $J_m$ and the state-conditioned correction map $\mathcal{M}_\theta$ depend on $u^{(m)}$.

\paragraph{Discussion (nonlinear).}
The loss is computed on \emph{linearized} systems along Newton trajectories, not directly on the nonlinear map. This has two effects:
\begin{enumerate}
\item It preserves the same preconditioned formulation as in the linear case.
\item It aligns training with actual inference, where the solver repeatedly tackles Jacobian systems with changing $u^{(m)}$.
\end{enumerate}
In practice, one may use an inner iteration
\begin{equation}
\delta u_{k+1}^{(m)}
=
\delta u_k^{(m)}
\;+\;
\mathcal{M}_\theta\!\left(
G\left(r^{(m)}-J_m\delta u_k^{(m)}\right)
\right),
\label{eq:newton_inner_iter}
\end{equation}
then set $\delta u^{(m)}\approx \delta u_{K}^{(m)}$ and update $u^{(m+1)}=u^{(m)}+\delta u^{(m)}$.

\section{Verification: Fast Transform Implementation of Green Operators}
\label{app:fft_verify_all}
We verify, case by case, when Green-operator application is fast. A concrete Fourier pseudospectral discretization recipe for $L_\eta$ and $G_\eta$ is provided in Appendix~\ref{app:fft_ps}.

\begin{proposition}[Transform-diagonalizable Green operators]
\label{prop:fft_all_cases}
On a uniform grid, if the chosen reference operator has constant coefficients and compatible boundary conditions (periodic, or transform-compatible Dirichlet/Neumann), then its discrete matrix is diagonalizable by a fast trigonometric transform. Consequently, applying the corresponding discrete Green operator costs $\mathcal{O}(N\log N)$.
\end{proposition}

\begin{proof}[Constructive verification]
Let $L_0$ be any constant-coefficient reference operator and $G_0=L_0^{-1}$. Under periodic boundary conditions, the discrete $L_0$ is block-circulant with circulant blocks, hence
\[
L_0 = F^{-1}\Lambda F,\qquad
G_0 = F^{-1}\Lambda^{-1}F,
\]
where $F$ is the multidimensional DFT matrix and $\Lambda$ contains the symbol values. Therefore
\[
G_0 q = \mathcal{F}^{-1}\!\left(\frac{\widehat{q}}{\lambda(\xi)}\right),
\]
implemented by FFT/IFFT in $\mathcal{O}(N\log N)$.
For Dirichlet or Neumann boundaries, the same statement holds with sine/cosine transforms (DST/DCT), which are FFT-based and have the same complexity order.
\end{proof}

\paragraph{Helmholtz case.}
With $L_\eta=-\Delta-(k_0^2+i\eta)$, the symbol is
\[
\lambda_{\mathrm{H}}(\xi)=|\xi|^2-(k_0^2+i\eta),
\]
so
\[
G_\eta q = \mathcal{F}^{-1}\!\left(\frac{\widehat{q}}{\lambda_{\mathrm{H}}(\xi)}\right).
\]
Hence the Born residual $G_\eta r$ is FFT-fast under the stated boundary setting.

\paragraph{CDR case.}
With $\mathcal{L}_{0,\mathrm{cdr}}=-\kappa_0\Delta+\mathbf{v}_0\cdot\nabla+\sigma_0$, the symbol is exactly \eqref{eq:cdr_G_symbol} denominator:
\[
\lambda_{\mathrm{cdr}}(\xi)=\kappa_0|\xi|^2+i\mathbf{v}_0\cdot\xi+\sigma_0.
\]
Thus
\[
G_{\mathrm{cdr}} q = \mathcal{F}^{-1}\!\left(\frac{\widehat{q}}{\lambda_{\mathrm{cdr}}(\xi)}\right),
\]
again $\mathcal{O}(N\log N)$.

\paragraph{Newton linear systems for nonlinear PDEs.}
For the Newton experiments, we use homogeneous Dirichlet boundaries and five-point FDM on interior nodes. The discrete Dirichlet Poisson reference $L_{\mathrm{D}}$ is diagonalized by the 2D DST-I. Writing
\[
L_{\mathrm{D}} = S^{-1}\Lambda_{\mathrm{N}}S,
\]
with $S$ the 2D DST-I transform, its inverse acts as
\[
G q = S^{-1}\!\left(\frac{Sq}{\lambda_{p,q}}\right),
\]
where the eigenvalues of $L_{\mathrm{D}}$ are
\[
\lambda_{p,q}
=
\frac{4}{h^2}\left[\sin^2\!\left(\frac{p\pi h}{2}\right)+\sin^2\!\left(\frac{q\pi h}{2}\right)\right].
\]
The Newton reference inverse is therefore applied by DST/IDST in $\mathcal{O}(N\log N)$.

\paragraph{Boundary-condition note for all cases.}
In scattering settings, we use an absorbing layer (e.g., sponge/PML-style damping) on a finite box and choose a \emph{constant-coefficient} reference operator on that computational domain; this is the operator inverted by FFT in the Born-inspired loss/iteration. If strict Dirichlet/Neumann boundaries are required, use DST/DCT (or FFT-based embedding), preserving near-identical complexity.
In particular, the experiments on Newton linear systems for nonlinear PDEs use homogeneous Dirichlet boundaries and DST-based inversion of the discrete Dirichlet Poisson reference $L_{\mathrm{D}}$.

\section{Fourier Pseudospectral Discretization of \texorpdfstring{$L_\eta$ and $G_\eta$}{L eta and G eta}}
\label{app:fft_ps}
This appendix gives the concrete discretization used when $L_\eta$ and $G_\eta=L_\eta^{-1}$ are implemented by FFT under periodic boundary conditions. In numerical-analysis terminology, this is a \emph{Fourier pseudospectral} (or Fourier collocation) discretization: derivatives are represented by exact Fourier multipliers, while coefficient multiplications are applied pointwise in physical space~\cite{boyd2001chebyshev,trefethen2000spectral,shen2011spectral}.

\paragraph{Periodic grid and Fourier modes.}
Let $\Omega=\prod_{\alpha=1}^d [0,L_\alpha)$ and use a uniform grid with $N_\alpha$ points in direction $\alpha$:
\[
x_{j_\alpha}=j_\alpha h_\alpha,\qquad
h_\alpha=L_\alpha/N_\alpha,\qquad
j_\alpha=0,\dots,N_\alpha-1.
\]
For grid field $u_j=u(x_j)$, define the DFT coefficients $\widehat u_n$ with multi-index
\[
n=(n_1,\dots,n_d),\qquad
n_\alpha\in\Big\{-\frac{N_\alpha}{2},\dots,\frac{N_\alpha}{2}-1\Big\},
\]
and physical wavenumbers
\[
\xi_{\alpha,n}=\frac{2\pi n_\alpha}{L_\alpha},\qquad
|\xi_n|^2=\sum_{\alpha=1}^d \xi_{\alpha,n}^2.
\]

\paragraph{Discrete symbol of $L_\eta$.}
For
\[
L_\eta=-\Delta-(k_0^2+i\eta),\qquad \eta>0,
\]
the Fourier symbol is
\[
\lambda_n = |\xi_n|^2-(k_0^2+i\eta),
\]
so
\[
\widehat{L_\eta u}_n=\lambda_n\,\widehat u_n.
\]
Hence $L_\eta$ is diagonal in Fourier space, and its inverse acts modewise:
\[
\widehat{(G_\eta q)}_n = \frac{\widehat q_n}{\lambda_n}.
\]
Because $\eta>0$, $\lambda_n$ has nonzero imaginary part and the modal inversion is well-defined.

\paragraph{FFT implementation of $G_\eta q$.}
Given $q$, compute:
\begin{enumerate}
\item $\widehat q=\mathcal{F}(q)$ (FFT),
\item $\widehat z_n=\widehat q_n/\lambda_n$ (pointwise complex division),
\item $z=\mathcal{F}^{-1}(\widehat z)$ (IFFT).
\end{enumerate}
Then $z=G_\eta q$. The complexity is $\mathcal{O}(N\log N)$ with $N=\prod_\alpha N_\alpha$.

\paragraph{Relation to $A$ and $V_\eta$.}
For heterogeneous Helmholtz,
\[
A u = -\Delta u - k(x)^2u,\qquad
V_\eta u=(k(x)^2-k_0^2-i\eta)u.
\]
In Fourier pseudospectral form:
\[
-\Delta u = \mathcal{F}^{-1}\!\big(|\xi|^2\widehat u\big),
\]
while $k(x)^2u$ and $V_\eta u$ are computed pointwise in physical space. This derivative-in-Fourier / product-in-physical split is exactly the pseudospectral pattern~\cite{boyd2001chebyshev,trefethen2000spectral}. For strongly non-smooth coefficients, standard de-aliasing (e.g., $2/3$ truncation) can be used~\cite{shen2011spectral}.

\paragraph{Boundary-condition scope.}
The FFT diagonalization above is exact for periodic boundaries. For homogeneous Dirichlet/Neumann boundaries, the same constant-coefficient reference operators are diagonalized by sine/cosine transforms (DST/DCT), with the same $\mathcal{O}(N\log N)$ complexity class. In this manuscript, $G_\eta$ in CBS and NPBS is implemented with the periodic/transform-compatible constant-coefficient reference operator to preserve fast application.

\section{Architectural Instantiations}
\label{app:architectures}
NPBS is not architecture-centered; it is compatible with different neural operators. For reproducibility, we briefly record the primary FNO instantiation used in most experiments.

\paragraph{FNO instantiation and design principle.}
Across Helmholtz, CDR, and Newton linear systems for nonlinear PDEs, we use an FNO-based preconditioner $\mathcal{M}_\theta$. Let $a$ denote PDE coefficient fields (e.g., $v$ in Helmholtz, or $a=(\kappa,\mathbf{v},\sigma)$ in CDR), and let $r$ denote the iterative residual. The design explicitly applies a \emph{nonlinear} transform to coefficient features and a \emph{linear} transform to residual features, then fuses both branches:
\begin{equation}
\mathcal{M}_\theta(r;a)
=
\mathcal{P} \Big(
\mathrm{LFL} \Big[
\mathrm{NLFL}(\Phi_\theta(a))
\odot
\mathrm{LFL}(\Psi_\theta(r))
\Big]
\Big).
\label{eq:fno_impl}
\end{equation}
The lifting/project mappings are
\begin{equation}
\Phi_\theta(a) = W_a * \mathrm{LayerNorm}(\mathrm{pad}(a)),
\qquad
\Psi_\theta(r) = W_r * \mathrm{pad}(r),
\label{eq:fno_lift}
\end{equation}
and
\begin{equation}
\mathcal{P}(z) = W_p * \mathrm{unpad}(z).
\label{eq:fno_proj}
\end{equation}
Here $\mathrm{NLFL}(\Phi_\theta(a))$ is used for nonlinear coefficient encoding, while the residual branch keeps linear processing through $\mathrm{LFL}(\Psi_\theta(r))$, ensuring linearity with respect to $r$.

The Fourier blocks are
\begin{equation}
\mathrm{LFL}(x)
=
\mathcal{F}^{-1}\!\big(\mathcal{W}^{\mathrm{F}}_{\theta}\odot \mathcal{F}(x)\big)
\;+
\mathcal{W}^{\mathrm{loc}}_{\theta} * x,
\label{eq:lfl}
\end{equation}
\begin{equation}
\mathrm{NLFL}(x)
=
\varphi\!\big(\mathrm{LFL}(x)\big),
\label{eq:nlfl}
\end{equation}
where $\mathcal{F}$ and $\mathcal{F}^{-1}$ are Fourier transforms, $\mathcal{W}^{\mathrm{F}}_{\theta}$ are truncated-mode spectral weights, $\mathcal{W}^{\mathrm{loc}}_{\theta}$ is a local $1\times1$ convolution, and $\varphi$ is GELU. In implementation, coefficient features pass through stacked NLFL blocks, while the residual branch uses a single LFL block.

\section{Dataset Illustration}
\label{app:dataset_figures}

\subsection{Heterogeneous Helmholtz Equation}
\label{app:helmholtz_dataset}

We use velocity models from the OpenFWI benchmark~\cite{deng2022openfwi}. For both CurveVel-A and CurveFault-B, the training set contains $3{,}000$ samples and the test set contains $50$ samples. The computational domain is discretized on a $256\times256$ grid with a 32-cell perfectly matched layer (PML) on each side, yielding a padded grid of size $320\times320$. Spatial velocity fields range from approximately $1500$ to $4500$\,m/s, and the grid resolution is approximately $12$ points per wavelength. Representative velocity models are shown in Figure~\ref{fig:openfwi_dataset}.

\begin{figure}[ht]
    \centering
    \includegraphics[width=0.95\linewidth]{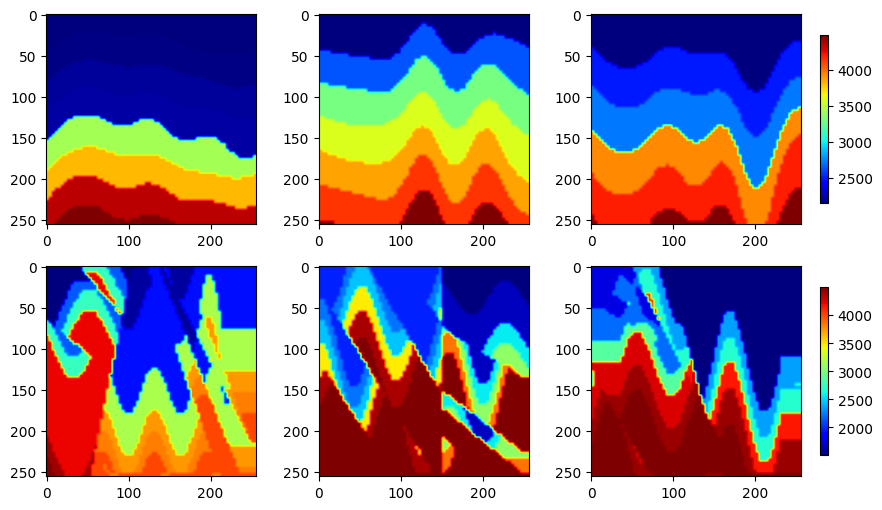}
    \caption{Representative heterogeneous velocity models from the OpenFWI datasets. Top row: CurveVel-A, featuring relatively smooth, curved velocity transitions. Bottom row: CurveFault-B, characterized by sharp discontinuities, strong contrasts, and complex fault structures.}
    \label{fig:openfwi_dataset}
\end{figure}

\subsection{CDR Dataset Generation}
\label{app:cdr_dataset}

The CDR benchmark uses a $128\times128$ periodic grid. The training set contains $4000$ samples and the test set contains $400$ samples. The diffusion field $\kappa$ is drawn from a Gaussian random field (GRF); the velocity $\mathbf{v}$ is obtained from a GRF stream function $\psi$ via
\begin{equation}
v_x=\frac{\partial \psi}{\partial y},\qquad v_y=-\frac{\partial \psi}{\partial x},
\end{equation}
so that incompressibility $\nabla \cdot \mathbf{v}=0$ is enforced by construction; the source term $f$ is a Gaussian bump with randomly sampled center and width. Representative samples are shown in Figure~\ref{fig:cdr_data_samples}.

\begin{figure}[ht]
    \centering
    \includegraphics[width=0.95\linewidth]{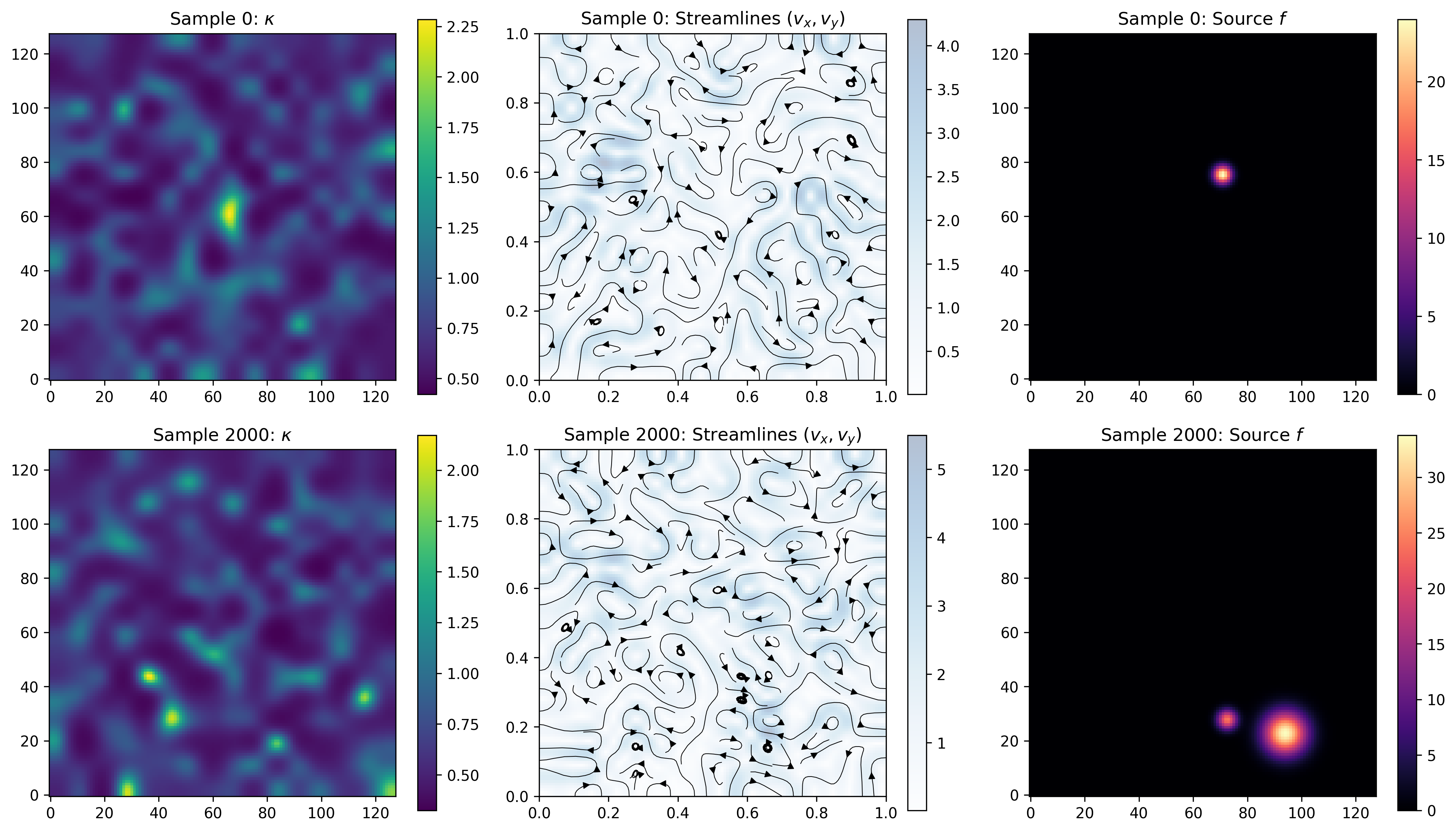}
    \caption{Representative samples from the CDR benchmark. Each row shows the diffusion field $\kappa$, the incompressible velocity field $\mathbf{v}$ visualized by streamlines, and the source term $f$.}
    \label{fig:cdr_data_samples}
\end{figure}

\subsection{Newton Benchmark Data}
\label{app:newton_dataset}

Following~\cite{hao2024newton}, the domain $\Omega=[0,1]^2$ is discretized with $N=64$ interior grid points (grid spacing $h=1/65$) and source amplitude $s=1600$. Training data consists of $10{,}000$ Newton trajectory pairs (combining two trajectory sets of $5{,}000$ each from different solution branches), and the test set contains $1{,}000$ samples. Supervision is obtained by generating Newton trajectories with sparse direct solves. To ensure that the trajectories probe distinct solution branches, initial guesses are constructed by perturbing a reference solution with low-frequency sine-series noise,
\begin{equation}
u^0=u^\star+\alpha\sum_{i,j=1}^{4}c_{ij}\sin(i\pi x)\sin(j\pi y),\quad
c_{ij}\sim\mathcal{N}(0,0.1^2),\ \alpha=4.0.
\end{equation}
Representative branches obtained in this way are shown in Figure~\ref{fig:newton_diff_sols}.

\begin{figure}[ht]
    \centering
    \includegraphics[width=0.65\linewidth]{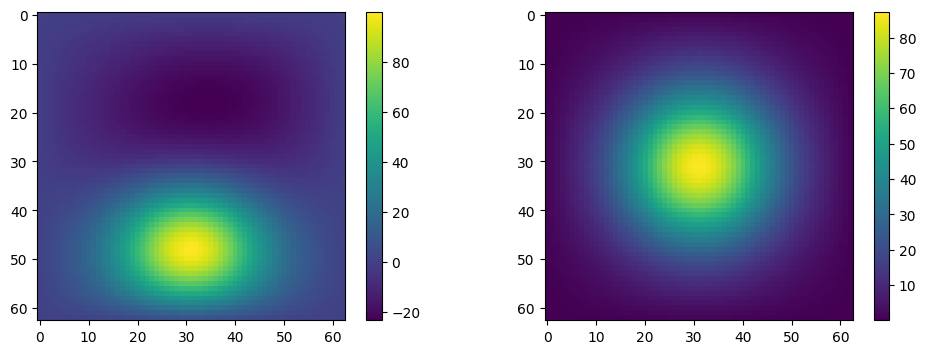}
    \caption{Multiple solutions of the nonlinear PDE under different initial guesses.}
    \label{fig:newton_diff_sols}
\end{figure}


\section{Experimental Setup}
\label{app:experimental_setup}

This appendix details the dataset splits, backbone architectures, and training hyperparameters for all experiments reported in Section~\ref{sec:numerical_results}.

\subsection{Heterogeneous Helmholtz Equation}
\label{app:setup_helmholtz}

\paragraph{Dataset.}
For the two-dimensional experiments, we use velocity models from the OpenFWI benchmark~\cite{deng2022openfwi}. For both CurveVel-A and CurveFault-B, the training set contains $3{,}000$ samples and the test set contains $50$ samples. The computational domain is discretized on a $256\times256$ grid with a 32-cell perfectly matched layer (PML) on each side, yielding a padded grid of size $320\times320$.

\paragraph{FNO backbone (primary).}
In the two-dimensional study, the primary preconditioner follows the FNO instantiation of Eq.~\eqref{eq:fno_impl} with $3$ NLFL layers, $25$ hidden channels, and $45$ retained Fourier modes per layer. The residual branch uses a single LFL with full-bandwidth modes ($N_{\mathrm{pad}}/2+1$). Training uses the AdamW optimizer for $400$ epochs with batch size $30$, initial learning rate $10^{-3}$, and a step-decay schedule (factor $0.8$ every $20$ epochs). Inference is performed in float64 with a relative tolerance of $10^{-6}$ and a maximum of $2{,}500$ iterations.

\paragraph{Comparison backbones.}
To verify that the NPBS gains are architecture-agnostic, we evaluate additional backbones under the same NPBS+$\mathcal{L}_{\mathrm{bs}}^{R_\eta}$ framework. In the two-dimensional comparisons, MgNO, U-Net, Encoder--solver, and MGCFNN are trained with AdamW for $200$ epochs, with batch size $30$, initial learning rate $10^{-3}$, and a step-decay schedule (factor $0.8$ every $10$ epochs). Their architecture-specific parameters are:
\begin{itemize}
    \item \textbf{MgNO}~\cite{he2024mgno}: A multigrid neural operator with $32$ channels, $4$-level V-cycle structure with $[2,2]$ pre-/post-smoothing iterations per level, and GELU activation.
    \item \textbf{MGCFNN}~\cite{xie2025mgcfnn}: A multigrid convolutional Fourier neural network with $32$ channels, $3$-level hierarchy. Level~0 uses spatial $5\times5$ convolutions (1 sweep); levels~1 and~2 use complex Fourier layers with wavenumber cutoffs $15$ and $10$ respectively (2 and 4 sweeps).
    \item \textbf{Encoder-Solver}~\cite{lerer2024multigrid}: A U-Net-style encoder-decoder with $32$ channels, $4$ levels, and implicit weight-sharing across decoder iterations.
    \item \textbf{U-Net}~\cite{ronneberger2015u}: A standard $4$-level U-Net with $32$ channels and skip connections; the rhs-path convolutions are bias-free to preserve linearity in the residual.
\end{itemize}
For runtime, we use \texttt{torch.compile} when stable. MGCFNN is evaluated without compilation because \texttt{torch.compile} led to divergence during inference. The original Encoder--solver diverged in our iterative tests; we therefore use a variant that is linear in the residual and retains its core multigrid-inspired components.

\paragraph{Three-dimensional Helmholtz setup.}
\label{app:3d_setup}
The OpenFWI Kimberlina-V1 velocity volumes have native resolution $350\times400\times400$ with $10$\,m grid spacing and a physical extent of approximately $3.5\times4\times4$\,km. We interpolate them to $192\times192\times192$ grids. The training set contains $600$ samples from the first $35$ years, while the test set contains $50$ samples from year $100$. We use MgNO with three-dimensional convolutions and FFTs, an absorbing-boundary width of $16$ cells, PPW $6$, and $200$ training epochs per formulation.

\subsection{Extensions Beyond the Helmholtz Equation}
\label{app:setup_generalization}

\subsubsection{Convection--Diffusion--Reaction (CDR)}
\label{app:setup_cdr}

\paragraph{Dataset.}
The CDR benchmark uses a $128\times128$ periodic grid. The training set contains $4{,}000$ samples and the test set contains $400$ samples. Coefficients are generated as described in Appendix~\ref{app:cdr_dataset}.

\paragraph{FNO backbone.}
The preconditioner uses the same FNO design principle as in the Helmholtz setting (Eq.~\eqref{eq:fno_impl}), with $25$ hidden channels, $3$ NLFL layers with $32$ retained Fourier modes per layer, and zero-padding of $8$ cells. Training uses Adam for $60$ epochs with batch size $32$, initial learning rate $10^{-3}$, and step-decay (factor $0.8$ every $3$ epochs). Inference uses a relative tolerance of $10^{-6}$.

\subsubsection{Newton Linear Systems for Nonlinear PDEs}
\label{app:setup_newton}

\paragraph{Dataset.}
Following~\cite{hao2024newton}, the domain $\Omega=[0,1]^2$ is discretized with $N=64$ interior grid points (grid spacing $h=1/65$) and the source amplitude is $s=1600$. Training data consists of $10{,}000$ Newton trajectory pairs (combining two trajectory sets of $5{,}000$ each from different solution branches), and the test set contains $1{,}000$ samples.

\paragraph{FNO backbone.}
The preconditioner uses $25$ hidden channels, $2$ NLFL layers with $20$ retained Fourier modes, and zero-padding of $8$ cells. Training uses Adam for $200$ epochs with batch size $50$, initial learning rate $10^{-4}$, and weight decay $10^{-6}$. The inner linear solver uses a relative tolerance of $10^{-4}$.


\section{Compute Resources}
\label{app:compute_resources}

All experiments were conducted using PyTorch 2.6.0 with CUDA 12.4 on a single workstation equipped with an Intel Xeon Platinum 8468 CPU and one NVIDIA H100 GPU.
All solvers---including classical baselines such as CBS and the shifted-Laplacian preconditioning---are implemented entirely in PyTorch and accelerated via CUDA, ensuring that wall-clock comparisons reflect algorithmic differences rather than implementation-level disparities.
No distributed training, large-scale pretraining, or fine-tuning of pretrained foundation models was involved; the main computational cost comes from training lightweight neural preconditioners and iterative solver inference.


\section{Training Convergence}
\label{app:training_convergence}

Figure~\ref{fig:train_val_loss_curvefault_b} shows the training and validation rTol histories for the three loss configurations on CurveFault-B. 

\begin{figure}[ht]
    \centering
    \includegraphics[width=0.75\linewidth]{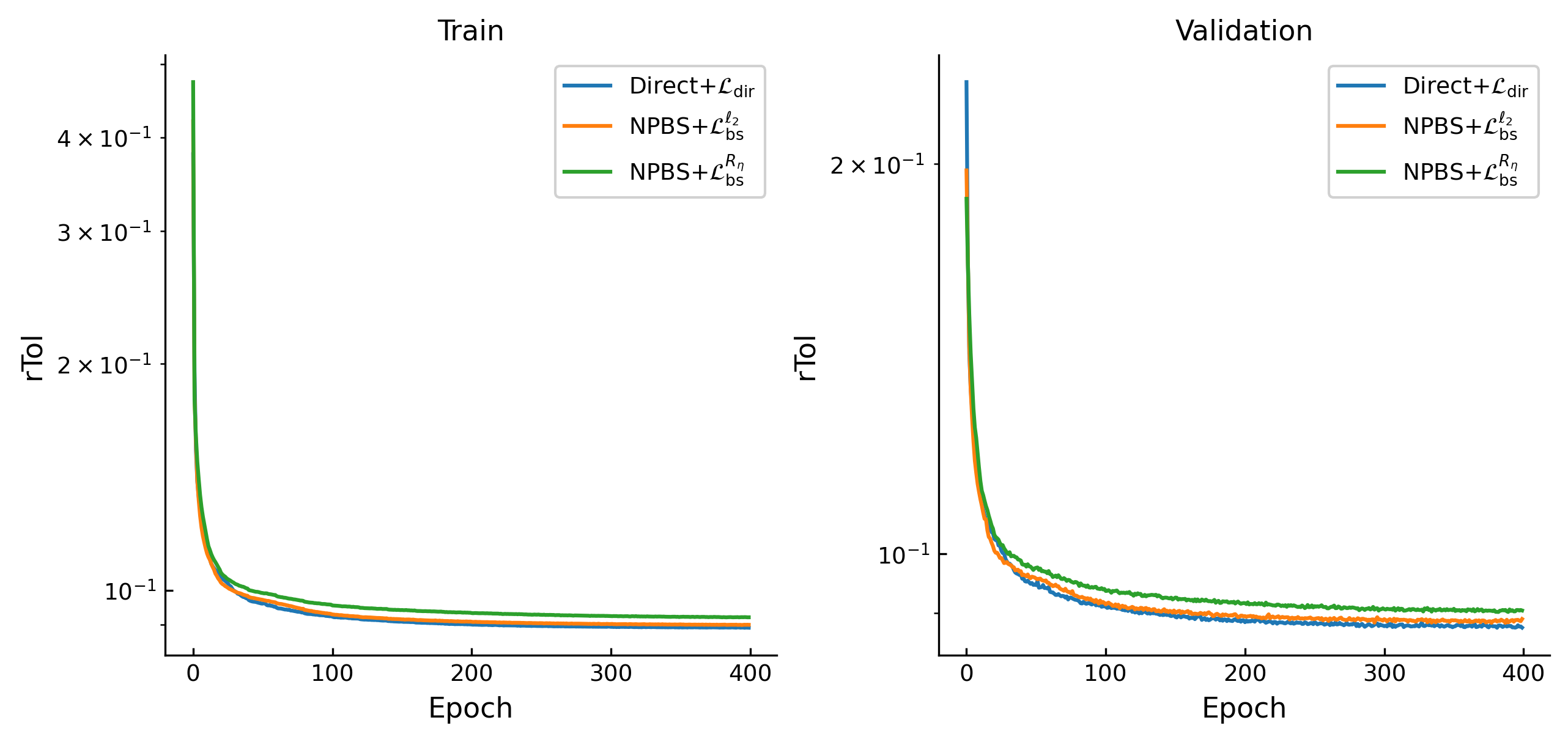}
    \caption{Training and validation rTol histories on CurveFault-B.}
    \label{fig:train_val_loss_curvefault_b}
\end{figure}

\begin{table}[htbp]
\centering
\caption{Training statistics on CurveFault-B under the same FNO setup.}
\label{tab:curvefault_b_train_stats}
\renewcommand{\arraystretch}{1.15}
\begin{tabular}{lccc}
\toprule
Method & Train rTol & Val rTol & Train time \\
\midrule
Direct + $\mathcal{L}_{\mathrm{dir}}$ & $8.91\times 10^{-2}$ & $8.77\times 10^{-2}$ & 52~min~18~s \\
NPBS + $\mathcal{L}_{\mathrm{bs}}^{\ell_2}$ & $8.99\times 10^{-2}$ & $8.90\times 10^{-2}$ & 52~min~35~s \\
NPBS + $\mathcal{L}_{\mathrm{bs}}^{R_\eta}$ & $9.20\times 10^{-2}$ & $9.04\times 10^{-2}$ & 52~min~39~s \\
\bottomrule
\end{tabular}
\end{table}

\section{Additional Helmholtz diagnostics and evaluation}
\label{app:rebuttal_diagnostics}

\subsection{Checking the spectral assumption}
\label{app:spectral_diagnostics}
Table~\ref{tab:operator_norm} reports the estimated operator norms for the two OpenFWI datasets. The diagnostic uses $1{,}000$ randomly selected samples and reports the mean estimate for each dataset and PPW. All reported means exceed one, so the sufficient condition for the spectral bounds in Appendix~\ref{app:conditional_spectral_bounds} cannot be invoked as a guarantee for these experiments. These are empirical norm estimates, not certified spectral bounds; the reported means do not assert that every sample has the same behavior. The algebraic identity and the induced training metric remain valid.
\begin{table}[htbp]
\centering
\caption{Estimated $\|G_\eta V_\eta\|_2$ (mean).}
\label{tab:operator_norm}
\begin{tabular}{lrrc}
\toprule
Dataset & PPW & Frequency (Hz) & $\|G_\eta V_\eta\|_2$ \\
\midrule
CurveFault-B & 12 & 92.0 & $1.3705$ \\
 & 14 & 78.9 & $1.3729$ \\
 & 16 & 69.0 & $1.3747$ \\
 & 24 & 46.0 & $1.3785$ \\
\midrule
CurveVel-A & 12 & 92.0 & $1.2881$ \\
 & 14 & 78.9 & $1.2952$ \\
 & 16 & 69.0 & $1.3020$ \\
 & 24 & 46.0 & $1.3246$ \\
\bottomrule
\end{tabular}
\end{table}

\subsection{Sensitivity to \texorpdfstring{$\eta$}{eta}}
\label{app:eta_sensitivity}
We use the canonical CBS scale $\eta_{\mathrm{base}}=\max_x|k(x)^2-k_0^2|$ as a reference and vary the ratio $\eta/\eta_{\mathrm{base}}$.
Table~\ref{tab:eta_sensitivity} reports the average NPBS iteration counts for the tested values.
The iteration diverges at $0.1\eta_{\mathrm{base}}$, whereas all reported scales from $0.2\eta_{\mathrm{base}}$ to $100\eta_{\mathrm{base}}$ converge.
The fewest iterations occur at $\eta=\eta_{\mathrm{base}}$; increasing the shift beyond this value progressively raises the iteration count.
These results support choosing $\eta$ near the canonical scale for this experiment.

\begin{table}[htbp]
\centering
\caption{Sensitivity of NPBS to the shift $\eta$. The symbol $\times$ denotes divergence at the tested value.}
\label{tab:eta_sensitivity}
\begin{tabular}{lrrrrrrrrrr}
\toprule
$\eta/\eta_{\mathrm{base}}$ & 0.1 & 0.2 & 0.4 & 0.8 & 1 & 1.4 & 2 & 5 & 10 & 100 \\
\midrule
Avg.\ Iters & $\times$ & 545 & 192 & 127 & \textbf{126} & 137 & 151 & 166 & 196 & 283 \\
\bottomrule
\end{tabular}
\end{table}

\subsection{Empirical contraction diagnostics}
The maximum contraction factors recorded for the MgNO tests are summarized in Table~\ref{tab:empirical_contraction}. Metric-matched NPBS has the smallest mean among the three formulations. These finite-test statistics do not establish a uniform contraction constant for all residuals or all test trajectories, and are distinct from the sufficient condition in \eqref{eq:learned_reduction}.
\begin{table}[htbp]
\centering
\caption{Recorded maximum contraction factors for MgNO (mean).}
\label{tab:empirical_contraction}
\begin{tabular}{lc}
\toprule
Formulation & $\rho_{\max}$ \\
\midrule
Direct + $\mathcal L_{\mathrm{dir}}$ & $0.9831$ \\
NPBS + $\mathcal L_{\mathrm{bs}}^{\ell_2}$ & $0.9547$ \\
NPBS + $\mathcal L_{\mathrm{bs}}^{R_\eta}$ & $\mathbf{0.9269}$ \\
\bottomrule
\end{tabular}
\end{table}

\subsection{Transfer under fixed-PPW scaling}
\label{app:resolution_transfer}
Table~\ref{tab:resolution_transfer} evaluates a single MgNO-based metric-matched NPBS model trained at $256^2$, with no retraining or fine-tuning on larger grids. The frequency is increased as the grid spacing decreases to retain PPW $12$. Thus the table measures transfer across structured-grid resolutions and frequencies, rather than refinement of one fixed-frequency PDE. The $256^2$ row is the reference point for this additional transfer evaluation. Table~\ref{tab:resolution_losses} compares all three training formulations in the same resolution study; the matched metric gives the lowest iteration counts at each reported resolution.
\begin{table}[htbp]
\centering
\caption{Metric-matched NPBS resolution transfer with MgNO and PPW $12$.}
\label{tab:resolution_transfer}
\begin{tabular}{rrrrr}
\toprule
Grid & $h$ (m) & Frequency (Hz) & Avg.\ Iters & Avg.\ Time \\
\midrule
$256\times256$ & 2.73 & 92 & 126 & 0.103\,s \\
$512\times512$ & 1.37 & 185 & 220 & 0.378\,s \\
$1024\times1024$ & 0.68 & 370 & 376 & 1.870\,s \\
$2048\times2048$ & 0.34 & 740 & 651 & 11.421\,s \\
\bottomrule
\end{tabular}
\end{table}
\begin{table}[htbp]
\centering
\caption{Iteration counts for the three formulations in the resolution study.}
\label{tab:resolution_losses}
\begin{tabular}{rccc}
\toprule
Grid & Direct + $\mathcal L_{\mathrm{dir}}$ & NPBS + $\mathcal L_{\mathrm{bs}}^{\ell_2}$ & NPBS + $\mathcal L_{\mathrm{bs}}^{R_\eta}$ \\
\midrule
$256\times256$ & 342 & 141 & \textbf{126} \\
$512\times512$ & 483 & 292 & \textbf{220} \\
$1024\times1024$ & 902 & 488 & \textbf{376} \\
\bottomrule
\end{tabular}
\end{table}

\FloatBarrier
\section{Limitations}
\label{app:limitations}

\paragraph{Dimensionality and grid structure.}
Our experiments cover two-dimensional PDEs and a three-dimensional Helmholtz benchmark on structured Cartesian grids. The current FFT/DST/DCT application of $G_\eta$ and the evaluated networks do not support general unstructured meshes. Extending them requires a suitable mesh-based reference inverse and compatible architecture. The fixed-PPW transfer tests change frequency with resolution and do not establish mesh-independent convergence for a fixed physical problem. General complex geometries and large-scale three-dimensional cost remain open questions.
\paragraph{Theoretical guarantees.}
The spectral bounds in Appendix~\ref{app:conditional_spectral_bounds} require $\|G_\eta V_\eta\|_2\le q<1$, whereas the reported experimental norm means exceed one. The learned-iteration convergence and perturbation bounds require the separate uniform residual-reduction condition \eqref{eq:learned_reduction}. We have not proved that training enforces this condition, and finite-test contraction statistics cannot certify it globally.

\paragraph{Applicability to PDE classes.}
NPBS is not suitable for arbitrary PDEs. It relies on a reference--perturbation split $A=L_0-V$ with an invertible reference operator $L_0$, so that the Born-preconditioned form $I-G_0V$ with $G_0=L_0^{-1}$ is well defined. This structure is natural for the Helmholtz, CDR, and Newton linear systems studied here, where a reference operator captures the main part of the PDE and the remaining coefficient or state dependence is treated as a perturbation. PDEs for which no meaningful or stable reference--perturbation split is available fall outside the present framework.

\end{document}